\documentclass[11pt,reqno]{amsart}
\usepackage{fullpage}
\usepackage{mathrsfs,amssymb,graphicx,verbatim,amsmath,amsfonts}
\usepackage{paralist}
\usepackage[breaklinks,pdfstartview=FitH]{hyperref}
\usepackage{upgreek}
\usepackage[mathscr]{euscript}
\renewcommand{\le}{\leqslant}
\renewcommand{\ge}{\geqslant}

\renewcommand{\gamma}{\upgamma}
\allowdisplaybreaks
\newcommand{\ud}[0]{\,\mathrm{d}}
\usepackage{esint}
\usepackage[autostyle]{csquotes}

\usepackage{dutchcal}

\renewcommand{\pi}{\uppi}

\newcommand{\dd}{\mathsf{d}}
\newcommand{\ee}{\mathsf{e}}

\newcommand{\n}{\{1,\ldots,n\}}

\newcommand{\f}{\upphi}

\renewcommand{\d}{\updelta}

\newcommand{\e}{\varepsilon}
\newcommand{\R}{\mathbb R}
\newcommand{\1}{\mathbf 1}

\newtheorem{theorem}{Theorem}

\newtheorem{corollary}[theorem]{Corollary}

\theoremstyle{remark}
\newtheorem{remark}[theorem]{Remark}

\newtheorem{conjecture}[theorem]{Conjecture}

\renewcommand{\tau}{\uptau}
\newcommand{\cM}{\mathcal{M}}
\newcommand{\cN}{\mathcal{N}}
\renewcommand{\xi}{\upxi}
\renewcommand{\rho}{\uprho}

\renewcommand{\subset}{\subseteq}
\newcommand{\C}{\mathbb C}

\newcommand{\N}{\mathbb N}

\newcommand{\eqdef}{\stackrel{\mathrm{def}}{=}}

\renewcommand{\theta}{\uptheta}
\renewcommand{\lambda}{\uplambda}

\DeclareMathOperator{\diam}{diam}

\newcommand{\A}{\mathsf{A}}

\newcommand{\G}{\mathsf{G}}
\renewcommand{\gamma}{\upgamma}
\renewcommand{\beta}{\upbeta}
\renewcommand{\alpha}{\upalpha}
\renewcommand{\kappa}{\upkappa}
\renewcommand{\psi}{\uppsi}
\renewcommand{\rho}{\uprho}
\renewcommand{\delta}{\updelta}
\renewcommand{\pi}{\uppi}
\renewcommand{\omega}{\upomega}
\renewcommand{\sigma}{\upsigma}
\renewcommand{\A}{\mathsf{A}}
\renewcommand{\eta}{\upeta}

\renewcommand{\kappa}{\upkappa}
\renewcommand{\mu}{\upmu}
\renewcommand{\nu}{\upnu}
\renewcommand{\pi}{\uppi}
\renewcommand{\zeta}{\upzeta}

\begin{document}

\title{A spectral gap precludes low-dimensional embeddings}

\author{Assaf Naor}
\address{Mathematics Department\\ Princeton University\\ Fine Hall, Washington Road, Princeton, NJ 08544-1000, USA}
\email{naor@math.princeton.edu}

\thanks{Supported by BSF grant 2010021, the Packard Foundation and the Simons Foundation.  The research presented here was conducted under the auspices of the Simons Algorithms and Geometry (A\&G) Think Tank.}


\date{\today}
\maketitle

\vspace{-0.25in}

\begin{abstract} We prove that there is a universal constant $C>0$ with the following property. Suppose that $n\in \N$ and that $\A=(a_{ij})\in M_n(\R)$ is a  symmetric stochastic matrix. Denote  the second-largest eigenvalue of $\A$ by $\lambda_2(\A)$. Then for {\em any} finite-dimensional normed space $(X,\|\cdot\|)$ we have
$$
\forall\, x_1,\ldots,x_n\in X,\qquad \dim(X)\ge \frac12 \exp\left(C\frac{1-\lambda_2(\A)}{\sqrt{n}}\bigg(\frac{\sum_{i=1}^n\sum_{j=1}^n\|x_i-x_j\|^2}{\sum_{i=1}^n\sum_{j=1}^na_{ij}\|x_i-x_j\|^2}\bigg)^{\frac12}\right).
$$
This implies that if an $n$-vertex $O(1)$-expander  embeds with average distortion $D\ge 1$ into $X$, then necessarily $\dim(X)\gtrsim n^{c/D}$ for some universal constant $c>0$, thus improving over the previously best-known estimate $\dim(X)\gtrsim (\log n)^2/D^2$ of Linial, London and Rabinovich,  strengthening a theorem of Matou\v{s}ek, and answering a question of Andoni, Nikolov, Razenshteyn and Waingarten.
\end{abstract}

\section{Introduction}

Given $n\in \N$ and a symmetric stochastic matrix $\A\in M_n(\R)$,  the eigenvalues of $\A$ will be denoted below by $1=\lambda_1(\A)\ge\ldots\ge \lambda_n(\A)\ge -1$. Here we prove the following statement.

\begin{theorem}\label{thm:main matrix intro}
There is a universal constant $C>0$ with the following property. Fix $n\in \N$ and a symmetric stochastic matrix $\A=(a_{ij})\in M_n(\R)$. For any finite-dimensional normed space $(X,\|\cdot\|)$,
\begin{equation}\label{eq:dimension lower intro}
\forall\, x_1,\ldots,x_n\in X,\qquad \dim(X)\ge \frac12 \exp\left(C\frac{1-\lambda_2(\A)}{\sqrt{n}}\bigg(\frac{\sum_{i=1}^n\sum_{j=1}^n\|x_i-x_j\|^2}{\sum_{i=1}^n\sum_{j=1}^na_{ij}\|x_i-x_j\|^2}\bigg)^{\frac12}\right).
\end{equation}
\end{theorem}
We shall next explain a noteworthy geometric consequence of Theorem~\ref{thm:main matrix intro} that arises from an examination of its special case when the matrix $\A$ is the normalized adjacency matrix of a connected graph. Before doing so, we briefly recall some standard terminology related to metric embeddings.

Suppose that $(\cM,d)$ is a finite metric space and $(X,\|\cdot\|)$ is a normed space. For $L\ge 0$, a mapping $\f:\cM\to X$ is said to be $L$-Lipschitz if $\|\f(x)-\f(y)\|\le Ld(x,y)$ for every $x,y\in \cM$. For $D\ge 1$, one says that $\cM$ embeds into $X$ with (bi-Lipschitz) distortion $D$ if there is a $D$-Lipschitz mapping $\f:\cM\to X$ such that $\|\f(x)-\f(y)\|\ge d(x,y)$ for every $x,y\in \cM$. Following Rabinovich~\cite{Rab08}, given $D\ge 1$ one says that $\cM$ embeds into $X$ with {average distortion} $D$ if there exists a $D$-Lipschitz mapping $\f:\cM\to X$ such that $\sum_{x,y\in \cM} \|\f(x)-\f(y)\|\ge \sum_{x,y\in \cM} d(x,y)$.

For $n\in \N$ write $[n]= \n$. Fix $k\in \{3,\ldots,n\}$ and let $\G=([n],E_\G)$ be a $k$-regular connected graph whose vertex set is $[n]$. The shortest-path metric that is induced by $\G$ on $[n]$ is denoted $d_\G:[n]\times [n]\to \N\cup\{0\}$. A simple (and standard) counting argument (e.g.~\cite{Mat97}) gives
\begin{equation}\label{eq:average dist bounded degree}
\frac{1}{n^2}\sum_{i=1}^n\sum_{j=1}^n d_\G(i,j)\gtrsim \frac{\log n}{\log k},
\end{equation}
where in~\eqref{eq:average dist bounded degree}, as well as in the rest of this article, we use the following (standard) asymptotic notation. Given two quantities $Q,Q'>0$, the notations
$Q\lesssim Q'$ and $Q'\gtrsim Q$ mean that $Q\le \mathsf{K}Q'$ for some
universal constant $\mathsf{K}>0$. The notation $Q\asymp Q'$
stands for $(Q\lesssim Q') \wedge  (Q'\lesssim Q)$. If  we need to allow for dependence on certain parameters, we indicate this by subscripts. For example, in the presence of an auxiliary parameter $\psi$, the notation $Q\lesssim_\psi Q'$ means that $Q\le c(\psi)Q' $, where $c(\psi) >0$ is allowed to depend only on $\psi$, and similarly for the notations $Q\gtrsim_\psi Q'$ and $Q\asymp_\psi Q'$.

 The normalized adjacency matrix of the graph $\G$, denoted $\A_\G$, is the matrix whose entry at $(i,j)\in [n]\times [n]$ is equal to $\frac{1}{k}\1_{\{i,j\}\in E_\G}$. Denote from now on $\lambda_2(\G)=\lambda_2(\A_\G)$. Let $(X,\|\cdot\|)$ be a finite-dimensional normed space. Fix $D\ge 1$ and  a mapping $\f:[n]\to X$ that satisfies
 \begin{equation}\label{eq:l2 gradient D}
\bigg(\frac{1}{|E_\G|}\sum_{\{i,j\}\in E_\G} \|\f(i)-\f(j)\|^2\bigg)^{\frac12}=\bigg(\frac{1}{n}\sum_{i=1}^n\sum_{j=1}^n (\A_\G)_{ij}\|\f(i)-\f(j)\|^2\bigg)^{\frac12}\le D.
  \end{equation}
 Condition~\eqref{eq:l2 gradient D} holds true, for example, if $\f$ is $D$-Lipschitz as a mapping from $([n],d_\G)$ to $(X,\|\cdot\|)$. Let $\eta>0$ be the implicit constant in the right hand side of~\eqref{eq:average dist bounded degree}, and suppose that $\f$ also satisfies
  \begin{equation}\label{eq:lower squared average dist expander}
  \bigg(\frac{1}{n^2}\sum_{i=1}^n\sum_{j=1}^n \|\f(i)-\f(j)\|^2\bigg)^{\frac12}\ge \eta\frac{\log n}{\log k}.
  \end{equation}
  Due to~\eqref{eq:average dist bounded degree} and the Cauchy--Schwarz inequality, conditions~\eqref{eq:l2 gradient D} and~\eqref{eq:lower squared average dist expander} hold true simultaneously (for an appropriately chosen $\f$) if e.g.~$([n],d_\G)$ embeds with average distortion $D/\eta$ into $(X,\|\cdot\|)$. At the same time, by an application of Theorem~\ref{thm:main matrix intro} with $x_i=\f(i)$ and $\A=\A_\G$ we see that
  $$
  \dim(X)\gtrsim e^{\frac{C\eta\left(1-\lambda_2(\A)\right)\log n}{D\log k} }=n^{\frac{C\eta\left(1-\lambda_2(\A)\right)}{D\log k} }.
  $$
  For ease of later reference, we record this conclusion as the following corollary.
\begin{corollary}\label{coro:expander non embed}
There exists a universal constant $\rho\in (0,\infty)$ such that for every $n\in \N$ and $k\in [n]$, if $\G=([n],E_\G)$ is a connected $n$-vertex $k$-regular graph and $D\ge 1$, then the dimension of any  normed space $(X,\|\cdot\|)$ into which the metric space $([n],d_\G)$ embeds with average distortion $D$ must satisfy $\dim(X)\gtrsim n^{c(\G)/D}$, where $c(\G)=\rho(1-\lambda_2(\A))/\log k$.
\end{corollary}

For every $n\in \N$ there exists a $4$-regular graph $\G_n=([n],E_{\G_n})$ with $\lambda_2(\G_n)\le 1-\d$, where $\d\in (0,1)$ is a universal constant; see the survey~\cite{HLW06} for this statement as well as much more on such {\em expander graphs}. It therefore follows from Corollary~\ref{coro:expander non embed} that for every $n\in \N$ there exists an $n$-point metric space $\cM_n$ with the property that its embeddability into any normed space with average distortion $D$ forces the dimension of that normed space to be at least $n^{c/D}$, where $c>0$ is a universal constant.\footnote{A good bound on the constant $c$  can be obtained if one applies Corollary~\ref{coro:expander non embed} to {\em Ramanujan graphs}~\cite{LPS88,Mar88} rather than to arbitrary expanders, but we shall not pursue this here.}  The significance of this statement will be discussed in Section~\ref{sec:dim reduction} below.

The desire to obtain Corollary~\ref{coro:expander non embed} was the goal that initiated our present investigation, because Corollary~\ref{coro:expander non embed} resolves (negatively) a question that was posed  by Andoni, Nikolov, Razenshteyn and Waingarten~\cite[Section~1.6]{ANRW16} in the context of their work on efficient approximate nearest neighbor search (NNS).  Specifically, they devised in~\cite{ANRW16}  an approach for proving a hardness result for NNS that requires the existence of an $n$-vertex expander that embeds with bi-Lipschitz distortion $O(1)$ into some normed space of dimension $n^{o(1)}$. Corollary~\ref{coro:expander non embed} shows that no such expander exists. One may view this statement as a weak indication that perhaps an algorithm for NNS in general norms could be designed with  better  performance than what is currently known, but we leave this interesting algorithmic question for future research and refer to~\cite{ANRW16}  for a full description of this connection. The previously best-known bound in the context of Corollary~\ref{coro:expander non embed} was due to Linial, London and Rabinovich in~\cite[Proposition~4.2]{LLR95}, where it was shown that if $\G$ is  $O(1)$-regular and $\lambda_2(\G)=1-\Omega(1)$, then any normed space $X$ into which $\G$ embeds with average distortion $D$ must satisfy $\dim(X)\gtrsim (\log n)^2/D^2$. The above exponential improvement over~\cite{LLR95}  is sharp, up to the value of  $c$, as shown by Johnson, Lindenstrauss and Schechtman~\cite{JLS87}.

\subsection{Dimensionality reduction}\label{sec:dim reduction} The present work relates to fundamental questions in pure and applied mathematics that have been extensively investigated over the past 3 decades, and are of major current importance. The overarching theme is that of {\em dimensionality reduction}, which corresponds to the desire to ``compress" $n$-point metric spaces using representations with few coordinates, namely embeddings into $\R^k$ with (hopefully) $k$ small, in such a way that pairwise distances could be (approximately) recovered by  computing  lengths in the image with respect to an appropriate norm on $\R^k$. Corollary~\ref{coro:expander non embed} asserts that this cannot be done in general if one aims for compression to $k=n^{o(1)}$ coordinates. In essence, it states that a spectral gap induces an inherent (power-type) high-dimensionality even if one allows for recovery of pairwise distances with large multiplicative errors, or even while only approximately preserving two averages of the squared distances:  along edges and all pairs, corresponding to~\eqref{eq:l2 gradient D} and~\eqref{eq:lower squared average dist expander}, respectively. In other words, we isolate two specific averages of pairwise squared distances of a finite collection of vectors in an arbitrary normed space, and show that if the ratio of these averages is roughly (i.e., up to a fixed but potentially large factor) the same as    in an expander then the dimension of the ambient space must be large.

In addition to obtaining specific results along these lines, there is need to develop techniques to address dimensionality questions that relate nonlinear (metric) considerations to the linear dimension of the vector space. Our main conceptual contribution is to exhibit a new approach to a line of investigations that previously yielded comparable results using algebraic techniques. In contrast, here we use an analytic method arising from a recently developed theory of nonlinear spectral gaps.

Adopting the terminology of~\cite[Definition~2.1]{LLR95}, given $D\in [1,\infty)$, $n\in \N$ and an $n$-point metric space $\cM$, define a quantity $\dim_D(\cM)\in \N$, called the (distortion-$D$) {\em metric dimension} of $\cM$, to be the minimum $k\in \N$ for which there exists a $k$-dimensional normed space $X_\cM$ such that $\cM$ embeds into $X_\cM$ with distortion $D$. We always have $\dim_D(\cM)\le \dim_1(\cM)\le n-1$ by the classical Fr\'echet isometric embedding~\cite{Fre06} into $\ell_\infty^{n-1}$. In their seminal work~\cite{JL84}, Johnson and Lindenstrauss asked~\cite[Problem~3]{JL84} whether   $\dim_D(\cM)=O(\log n)$ for some $D=O(1)$ and every $n$-point metric space $\cM$. Observe that the $O(\log n)$ bound arises naturally here, as it cannot be improved due to a standard volumetric argument when one considers embeddings of the $n$-point equilateral space; see also Remark~\ref{rem:ribe} below for background on  the Johnson--Lindenstrauss question in the context of the Ribe program. Nevertheless, Bourgain proved~\cite[Corollary~4]{Bou85} that this question has a negative answer. He showed that for arbitrarily large $n\in \N$ there is  an $n$-point metric space $\cM_n$ such that $\dim_D(\cM)\gtrsim (\log n)^2/(D\log\log n)^2$ for every $D\in [1,\infty)$. He also posed in~\cite{Bou85}  the natural question of determining the asymptotic behavior of the maximum of $\dim_D(\cM)$ over all $n$-point metric spaces $\cM$. It took over a decade for this question to be  resolved.

In terms of upper bounds, Johnson, Lindenstrauss and Schechtman~\cite{JLS87} proved that there exists a universal constant $\alpha>0$ such that for every $D\ge 1$ and $n\in \N$ we have $\dim_D(\cM)\lesssim_D n^{\alpha/D}$ for any $n$-point metric space $\cM$.  In~\cite{Mat92,Mat96}, Matou\v{s}ek improved this result by showing that one can actually embed $\cM$ with distortion $D$ into $\ell_\infty^k$ for some $k\in \N$ satisfying $k\lesssim_D n^{\alpha/D}$, i.e., the target normed space need not depend on $\cM$ (Matou\v{s}ek's proof is also simpler than that of~\cite{JLS87}, and it yields a smaller value of $\alpha$; see the exposition in Chapter~15 of the monograph~\cite{Mat02}).

In terms of lower bounds, an asymptotic improvement over~\cite{Bou85} was made by Linial, London and Rabinovich~\cite[Proposition~4.2]{LLR95}, who showed that  for arbitrarily large $n\in \N$ there exists an $n$-point metric space $\cM_n$ such that $\dim_D(\cM_n)\gtrsim (\log n)^2/D^2$ for every $D\in [1,\infty)$. For small distortions, Arias-de-Reyna and Rodr{\'{\i}}guez-Piazza proved~\cite{AR92}  the satisfactory assertion that for arbitrarily large $n\in \N$ there exists an $n$-point metric space $\cM_n$ such that $\dim_D(\cM_n) \gtrsim_D n$ for every $1\le D<2$. For larger distortions, it was asked in~\cite[page~109]{AR92} whether for every $D\in (2,\infty)$ and $n\in \N$ we have $\dim_D(\cM)\lesssim_D (\log n)^{O(1)}$ for any $n$-point metric space $\cM$.   In~\cite{Mat96}, Matou\v{s}ek famously answered  this question negatively by proving Theorem~\ref{thm:matousek} below via a clever argument that relies on (a modification of) graphs of large girth with many edges and an existential counting argument (inspired by ideas of Alon, Frankl and R\"odl~\cite{AFR85})  that uses the classical theorem of  Milnor~\cite{Mil64} and Thom~\cite{Tho65} from real algebraic geometry.
\begin{theorem}[Matou\v{s}ek~\cite{Mat96}]\label{thm:matousek} For every $D\ge 1$ and arbitrarily large $n\in \N$, there is  an $n$-point metric space $\cM_{n}(D)$ such that $\dim_D\big(\cM_{n}(D)\big)\gtrsim_D n^{c/D}$, where $c>0$ is a universal constant.
 \end{theorem}
 Due to the upper bound that was quoted above, Matou\v{s}ek's theorem satisfactorily  answers the questions of Johnson--Lindenstrauss and Bourgain, up to the value of the universal constant $c$. Corollary~\ref{coro:expander non embed} also resolves these questions, via an approach for deducing dimensionality lower bounds from rough (bi-Lipschitz) metric information that differs markedly from Matou\v{s}ek's argument.

 Our solution has some new features. The spaces $\cM_n(D)$ of Theorem~\ref{thm:matousek} can actually be taken to be independent of the distortion $D$, while the construction of~\cite{Mat96} depends on $D$ (it is based on graphs of girth of order $D$). One could alternatively achieve this by considering the disjoint union of the spaces $\{\cM_n(2^k)\}_{k=0}^{m}$ for  $m\asymp \log n$, which is a metric space of size $O(n\log n)$. More importantly, rather than using an ad-hoc construction (relying also on a non-constructive existential statement) as in~\cite{Mat96}, here we specify a natural class of metric spaces, namely the shortest-path metrics on  expanders (see also Remark~\ref{rem:quotients} below), for which Theorem~\ref{thm:matousek} holds. Obtaining this result for this concrete class of metric spaces is needed to answer the question of~\cite{ANRW16} that was quoted above. Finally, Matou\v{s}ek's approach based on the Millnor--Thom theorem uses the fact that the embedding has controlled bi-Lipschitz distortion, while our approach is robust in the sense that it deduces the stated lower bound on the dimension from an embedding with small average distortion.

\begin{remark}\label{rem:ribe} The {\em Ribe program} aims to uncover an explicit ``dictionary" between the local theory of Banach spaces and general metric spaces, inspired by an important rigidity theorem of Ribe~\cite{Rib76} that indicates that a dictionary of this sort should exist. See the introduction of~\cite{Bou86} as well as the surveys~\cite{Kal08,Nao12,Bal13} and the monograph~\cite{Ost13} for more on this topic. While more recent research on dimensionality reduction is most often motivated by the need to compress data, the initial motivation of the question of Johnson and Lindenstrauss~\cite{JL84} that we quoted above arose from the Ribe program. It seems simplest to include here a direct quotation of Matou\v{s}ek's explanation in~\cite[page~334]{Mat96} for the origin of the investigations that led to Theorem~\ref{thm:matousek}.
\blockquote{\em ...This investigation started in the context of the
local Banach space theory, where the general idea was to obtain some analogs for
general metric spaces of notions and results dealing with the structure of finite
dimensional subspaces of Banach spaces. The distortion of a mapping should
play the role of the norm of a linear operator, and the quantity $\log n$, where $n$ is
the number of points in a metric space, would serve as an analog of the dimension
of a normed space. Parts of this programme have been carried out by Bourgain,
Johnson, Lindenstrauss, Milman and others...}
Despite many previous successes of the Ribe program, not all of the questions that it raised turned out to have a positive answer (see e.g.~\cite{MN13-convexity}).  Theorem~\ref{thm:matousek} is among the most extreme examples of failures of natural steps in the Ribe program, with the final answer being exponentially worse than the initial predictions.  Corollary~\ref{coro:expander non embed} provides a further explanation of this phenomenon.
\end{remark}

\begin{remark}\label{rem:quotients}  The reasoning prior to Corollary~\ref{coro:expander non embed} gives  the following statement that applies to regular graphs that need not have bounded degree. Fix $\beta>0$ and $n\in \N$. Suppose that $\G=([n],E_\G)$ is a connected regular graph that satisfies $(1-\lambda_2(\G))\sum_{i=1}^n\sum_{j=1}^n d_\G(i,j)\ge \beta n^2\log n$. Then, $\dim_D(\G)\gtrsim n^{C\beta/D}$ for every $D\ge 1$, where $C>0$ is the universal constant of Theorem~\ref{thm:main matrix intro} and we use the notation $\dim_D([n],d_\G)=\dim_D(\G)$. Let $\diam(\G)$ be the diameter of $([n],d_\G)$ and suppose (for simplicity) that $\G$ is vertex-transitive (e.g., $\G$ can be the Cayley graph of a finite group). Then, it is  simple to check that $n^2\diam(\G)\ge \sum_{i=1}^n\sum_{j=1}^n d_\G(i,j)\ge n^2\diam(\G)/4$ (see. e.g.~equation (4.24) in~\cite{Nao14}), and therefore the above reasoning shows that every vertex-transitive graph satisfies
\begin{equation}\label{eq:woith diam}
\forall\, D\ge1,\qquad \dim_D(\G)\gtrsim e^{\frac{C}{4D}(1-\lambda_2(\G))\diam(\G)}.
 \end{equation}
In particular, it follows from~\eqref{eq:woith diam} that if $([n],d_\G)$ embeds with distortion $O(1)$ into some normed space of dimension $(\log n)^{O(1)}$, then necessarily $(1-\lambda_2(\G))\diam(\G)\lesssim \log\log n$.

 There are many examples of Cayley graphs $\G=([n],E_\G)$ for which $\lambda_2(\G)=1-\Omega(1)$ and $\diam(\G)\gtrsim \log n$ (see e.g.~\cite{AR94,NR09}). In all such examples, \eqref{eq:woith diam} asserts that $\dim_D(\G)\gtrsim n^{c/D}$ for some universal constant $c>0$. The Cayley graph that was studied in~\cite{KN06} (a quotient of the Hamming cube by a good code) now shows that there exist arbitrarily large $n$-point metric spaces $\cM_n$ with $\dim_1(\cM_n)\lesssim \log n$  (indeed, $\cM_n$ embeds isometrically into $\ell_1^{k}$ for some $k\lesssim \log n$), yet $\cM_n$ has a $O(1)$-Lipschitz quotient (see~\cite{BJLPS99} for the relevant definition)  that does not embed with distortion $O(1)$ into any normed space of dimension $n^{o(1)}$. To the best of our knowledge, it wasn't previously known that the metric dimension $\dim_D(\cdot)$ can become asymptotically larger (and even increase exponentially)  under Lipschitz quotients, which is yet another major departure from the linear theory, in contrast to what one would normally predict in the context of the Ribe program.
\end{remark}

\subsection{Roadmap} Theorem~\ref{thm:main matrix intro} will be proven in Section~\ref{sec:proof main}, which starts with an informal overview of the main ideas that enter into the proof. Section~\ref{sec:average dist}  derives an additional example of an application of these ideas to metric embedding theory. We end with Section~\ref{sec:discussion}, which contains  further discussion  about dimensionality reduction questions and presents some important open problems.

\subsection*{Acknowledgements} I am grateful to Alex Andoni and Ilya Razenshteyn for encouraging me to work on the question that is addressed here. I also thank Gideon Schechtman for helpful discussions.

\section{Proof of Theorem~\ref{thm:main matrix intro}}\label{sec:proof main}

 Modulo the use of a  theorem about nonlinear spectral gaps which is a main result of~\cite{Nao14}, our proof  of Theorem~\ref{thm:main matrix intro} is not long. We rely here on an argument that  perturbs any finite-dimensional normed space (by complex interpolation with its distance ellipsoid)  so as to make the result of~\cite{Nao14} become applicable, and we  proceed to show that by optimizing over the size of the perturbation one can deduce the desired dimensionality-reduction lower bound. This idea is the main conceptual contribution of the present work, and we derive an additional application of it to embedding theory in Section~\ref{sec:average dist} below.  We begin with an informal overview of this argument.

\subsection{Overview}\label{sec:overview} The precursors of our approach are the works~\cite{LN04-diamond} and~\cite{LMN05} about the impossibility of dimensionality reduction in $\ell_1$ and $\ell_\infty$, respectively. It was shown in~\cite{LN04-diamond} (respectively~\cite{LMN05}) that a certain $n$-point metric space $\cM_1$ (respectively $\cM_\infty$) does not admit a low-distortion embedding into $X=\ell_1^k$ (respectively $X=\ell_\infty^k$) with $k$ small, by arguing that if $k$ were indeed small then there would be  a normed space $Y$ that is ``close" to $X$, yet any embedding of $\cM_1$ (respectively $\cM_\infty$) into $Y$ incurs large distortion. This leads to a contradiction, provided that the assumed embedding of $\cM_1$ (respectively $\cM_\infty$) into $X$ had sufficiently small distortion relative to the closeness of $Y$ to $X$. In the setting of~\cite{LN04-diamond,LMN05}, there is a natural one-parameter family of normed spaces that tends to $X$, namely the spaces $\ell_p^k$ with $p\to 1$ or $p\to \infty$, respectively, and indeed the space $Y$ is taken to be an appropriate member of this  family. For a general normed space $X$, it is a priori unclear how to perturb it so as to implement this strategy. Moreover, the arguments of~\cite{LN04-diamond,LMN05} rely on additional special properties of the specific normed spaces in question that hinder their applicability to general normed spaces: The example of~\cite{LN04-diamond} is unsuited to the question that we study here because in was shown in~\cite{KLMN05} that in fact $\dim_D(\cM_1)\lesssim \log n$ for some $D=O(1)$\footnote{Specifically,  the space considered in~\cite{LN04-diamond} was shown in~\cite{KLMN05} to embed with distortion $O(1)$ into $\ell_\infty^{O(\log n)}$, and by~\cite{Rab08} it even embeds with average distortion $O(1)$ into the real line.}; and, the proof in~\cite{LMN05} of the non-embeddability of $\cM_\infty$ into $Y$ is based on a theorem of Matou\v{s}ek~\cite{Mat97} whose proof relies heavily on the coordinate structure of $Y=\ell_p^k$. We shall overcome the former difficulty  by using the complex interpolation method to perturb $X$, and we shall overcome the latter difficulty by invoking the theory of nonlinear spectral gaps.

Suppose that $(X,\|\cdot\|)$ is a finite-dimensional normed space. The perturbative step of our argument considers the Hilbert space $H$ whose unit ball is an ellipsoid that is closest to the unit ball of $X$, i.e., a {\em distance ellipsoid} of $X$; see Section~\ref{sec:dist ellipsoid} below. We then use the complex interpolation method (see Section~\ref{sec:interpolation} below) to obtain a one-parameter family of normed spaces $\{[X_\C,H_\C]_\theta\}_{\theta\in [0,1]}$ that intertwines the complexifications (see Section~\ref{sec:complexification} below) of $X$ and $H$, respectively. These intermediate spaces will serve as a proxy for the one-parameter family $\{\ell_p^n\}_{p\in [1,\infty]}$ that was used in~\cite{LMN05}. In order to see how they fit into this picture we briefly recall the argument of~\cite{LMN05}.

Suppose that $\G=([n],E_\G)$ is a $O(1)$-regular graph with $\lambda_2(\G)=1-\Omega(1)$ (i.e., an expander). In~\cite[Proposition~4.1]{LMN05} it was shown that for every $D\ge 1$ and $k\in \N$,  if $([n],d_\G)$ embeds with distortion $D$ into $\ell_\infty^k$, then necessarily $k\gtrsim n^{c/D}$ for some universal constant $c>0$. This is so because Matou\v{s}ek proved in~\cite{Mat97} that for any $p\in [1,\infty)$, any embedding of $([n],d_\G)$ into $\ell_p$ incurs distortion at least  $\eta(\log n)/p$, where $\eta>0$ is a universal constant. The norms on  $\ell_\infty^k$ and $\ell_{\log k}^k$ are within a factor of $e$ of each other, so it follows that $D\ge \eta(\log n)/(e\log k)$, i.e., $k\ge n^{\eta/(e D)}$.

The reason for the distortion lower bound of~\cite{Mat97} that was used above is that~\cite{Mat97}  shows that there exists a universal constant $C>0$ such that for every $p\ge 1$ we have
\begin{equation}\label{eq:mat extrapolation}
\forall\, t_1,\ldots,t_n\in \R,\qquad \frac{1}{n^2} \sum_{i=1}^n\sum_{j=1}^n |t_i-t_j|^p\le \frac{(Cp)^p}{|E_\G|}\sum_{\{i,j\}\in E_\G} |t_i-t_j|^p.
\end{equation}
The proof of~\eqref{eq:mat extrapolation} relies on the fact that the case $p=2$ of~\eqref{eq:mat extrapolation} is nothing more than the usual Poincar\'e inequality that follows through elementary linear algebra from the fact that $\lambda_2(\G)$ is bounded away from $1$, combined with an extrapolation argument that uses elementary inequalities for real numbers (see also the expositions in~\cite{BLMN05,NS11}). By summing~\eqref{eq:mat extrapolation} over coordinates we deduce that
\begin{equation}\label{eq:lp matousek}
\forall\, x_1,\ldots,x_n\in \ell_p,\qquad \bigg(\frac{1}{n^2} \sum_{i=1}^n\sum_{j=1}^n \|x_i-x_j\|_p^p\bigg)^{\frac{1}{p}}\lesssim p\bigg(\frac{1}{|E_\G|}\sum_{\{i,j\}\in E_\G} \|x_i-x_j\|_p^p\bigg)^{\frac{1}{p}}.
\end{equation}
This implies that any embedding of $([n],d_\G)$ into $\ell_p$ incurs average distortion at least a constant multiple of $(\log n)/p$ via the same reasoning as  the one that preceded Corollary~\ref{coro:expander non embed}.

The reliance on coordinate-wise inequalities in the derivation of~\eqref{eq:lp matousek} is problematic when it comes to the need to treat a general finite-dimensional normed space $(X,|\cdot\|)$. This ``scalar" way of reasoning also leads to the fact that in~\eqref{eq:lp matousek} the $\ell_p$ norm is raised to the power $p$. Since, even in the special case $X=\ell_p^{k}$, \eqref{eq:lp matousek}  is applied in the above argument when $p=\log \dim(X)$, this hinders our ability to deduce an estimate such as the conclusion~\eqref{eq:dimension lower intro} of Theorem~\ref{thm:main matrix intro}.

To overcome this obstacle, we consider a truly nonlinear (quadratic) variant of~\eqref{eq:lp matousek} which is known as a {\em nonlinear spectral-gap inequality}. See Section~\ref{sec:gap} below for the formulation of this concept, based on a line of works in metric geometry that has been more recently investigated systematically  in~\cite{MN13,MN14,Nao14,MN15}. Our main tool is a result of~\cite{Nao14}, which is quoted as Theorem~\ref{thm:quote nao14}  below. It provides an estimate in the spirit of~\eqref{eq:lp matousek} for $n$-tuples of vectors in each of the complex interpolation spaces $\{[X_\C,H_\C]_\theta\}_{\theta\in (0,1]}$, in terms of the parameter $\theta$ and the $p$-smoothness constant of the normed space $[X_\C,H_\C]_\theta$ (see Section~\ref{sec:smoothness} below for the relevant definition). We then implement the above perturbative strategy  by estimating the closeness of $X$ to a subspace of $[X_\C,H_\C]_\theta$, and optimizing over the auxiliary interpolation parameter $\theta$.

While the result of~\cite{Nao14} that we use here is substantial, we  encourage readers to examine its proof rather than relying on it as a ``black box," because we believe that this proof is illuminating and accessible to non-experts. Specifically, the proof in~\cite{Nao14} of Theorem~\ref{thm:quote nao14}  below relies on  Ball's notion of Markov type~\cite{Bal92} $p$ through the martingale method of~\cite{NPSS06}, in combination with complex interpolation and a trick of V.~Lafforgue that was used by Pisier in~\cite{Pis10}. It is interesting to observe that here we use  the fact that bound that is obtained in~\cite{Nao14} depends on the $p$-smoothness constant of    $[X_\C,H_\C]_\theta$, but it contains  no other dependence on $p$. Since in our final optimization over $\theta$ we take $p$ to be very close to $1$, we can't allow for an implicit dependence on $p$ that is unbounded as $p\to 1$. Such a $p$-independent bound is indeed obtained in~\cite{Nao14}, but unlike the present application, it was a side issue in~\cite{Nao14}, where only the case $p=2$ was used.

\subsection{Distance ellipsoids}\label{sec:dist ellipsoid} Recall that given $\dd\in [1,\infty)$, a Banach space $(X,\|\cdot\|)$ is said to be $\dd$-isomorphic to a Hilbert space if it admits  a scalar product $\langle\cdot,\cdot\rangle:X\times X\to \R$, such that if we denote its associated Hilbertian norm by $|x|=\sqrt{\langle x,x\rangle}$, then
\begin{equation}\label{eq:d-isomorph}
\forall\, x\in X,\qquad |x|\le \|x\|\le \dd |x|.
\end{equation}
 The (Banach--Mazur) {\em Euclidean distance} of $X$, denoted $\dd_X\in [1,\infty)$, is then defined to be the infimum over those $\dd\in [1,\infty)$ for which~\eqref{eq:d-isomorph} holds true.  If $X$ is not $\dd$-isomorphic to a Hilbert space for any $\dd\in [1,\infty)$, then we write $\dd_X=\infty$. If $X$ is finite-dimensional, then John's theorem~\cite{Joh48} asserts that $\dd_X\le \sqrt{\dim(X)}$ (and, in many settings asymptotically better bounds on $\dd_X$ in terms of $\dim(X)$ are known; see~\cite{MW78,Tom89}).   By a standard compactness argument, if $X$ is finite-dimensional, then the infimum in the definition of $\dd_X$ is attained. In that case, the unit ball of the Hilbertian norm $|\cdot|$, i.e., the set $\{x\in X:\ |x|\le 1\}$, is commonly called a {\em distance ellipsoid} of $X$. Note that the distance ellipsoid need not be unique; see~\cite{Pra02} for more on this topic.

\subsection{Nonlinear spectral gaps}\label{sec:gap} Suppose that $(\cM,d_\cM)$ is a metric space, $n\in \N$ and $p\in (0,\infty)$. Following~\cite{MN14}, the (reciprocal of) the {\em nonlinear spectral gap} with respect to $d_{\cM}^p$ of a symmetric stochastic matrix $\A=(a_{ij})\in M_n(\R)$, denoted $\gamma(\A,d_\cM^p)$, is the smallest $\gamma\in (0,\infty)$ such that
$$
\forall\,x_1,\ldots,x_n\in \cM,\qquad \frac{1}{n^2}\sum_{i=1}^n \sum_{j=1}^n d_{\cM}(x_i,x_j)^p\le \frac{\gamma}{n} \sum_{i=1}^n \sum_{j=1}^n a_{ij}d_{\cM}(x_i,x_j)^p.
$$
We refer to~\cite{MN14} for an extensive discussion of this notion; it suffices to state here that the reason for this nomenclature is that if we denote the standard metric on the real line by $d_\R$  (i.e., $d_\R(s,t)=|s-t|$ for every $s,t\in \R$), then it is straightforward to check that $\gamma(\A,d_\R^2)=1/(1-\lambda_2(\A))$.

In general, nonlinear spectral gaps can differ markedly from the usual (reciprocal of) the gap in the (linear) spectrum, though~\cite{Nao14} is devoted to an investigation of various settings in which one can obtain comparison inequalities for nonlinear spectral gaps when the underlying metric is changed. Estimates on $\gamma(\A,d_\cM^p)$ have a variety of applications in metric geometry, and  here we establish their relevance to dimensionality reduction. Specifically, we shall derive below the following result, which  will be shown to imply Theorem~\ref{thm:main matrix intro}.

\begin{theorem}[Nonlinear spectral gap for Hilbert isomorphs]\label{th:isomorph}  Fix $n\in \N$ and a symmetric stochastic matrix $\A=(a_{ij})\in M_n(\R)$. Then for every normed space $(X,\|\cdot\|)$ with $\dd_X<\infty$, we have
\begin{equation}\label{eq:cases optimized gamma}
\gamma\big(\A,\|\cdot\|^2\big)\lesssim \left\{\begin{array}{ll}
\frac{\dd_X^2}{1-\lambda_2(\A)}&\mathrm{if\ }\  \dd_X\sqrt{1-\lambda_2(\A)}\le e,\\
\left(\frac{\log\left(\dd_X\sqrt{1-\lambda_2(\A)}\right)}{1-\lambda_2(\A)}\right)^2&\mathrm{if\ }\  \dd_X\sqrt{1-\lambda_2(\A)}> e.\end{array}\right.
\end{equation}
\end{theorem}

\begin{proof}[Proof of Theorem~\ref{thm:main matrix intro} assuming Theorem~\ref{th:isomorph}] We claim that~\eqref{eq:cases optimized gamma} implies the following simpler bound.
\begin{equation}\label{eq:gamma hilbert isomorph}
\gamma\big(\A,\|\cdot\|^2\big)\lesssim \left(\frac{\log \left(\dd_X\sqrt{2}\right)}{1-\lambda_2(\A)}\right)^2.
\end{equation}
Indeed, if $\dd_X\sqrt{1-\lambda_2(\A)}> e$, then the right hand side of~\eqref{eq:gamma hilbert isomorph} is at least the right hand side of~\eqref{eq:cases optimized gamma}   due to the fact that, since $\A$ is symmetric and stochastic, $\lambda_2(\A)\ge -1$, so that $\sqrt{1-\lambda_2(\A)}\le \sqrt{2}$. On the other hand, if $\dd_X\sqrt{1-\lambda_2(\A)}\le e$ then $\dd_X^2/(1-\lambda_2(\A))\le e^2/(1-\lambda_2(\A))^2$, which is at most a universal constant multiple of the right hand side of~\eqref{eq:gamma hilbert isomorph} because $\dd_X\ge 1$.

By the definition of $\gamma\big(\A,\|\cdot\|^2\big)$, it follows from~\eqref{eq:gamma hilbert isomorph} that there exists a universal constant $\alpha>0$ such that for every $x_1,\ldots,x_n\in X$ we have
$$
\frac{1}{n^2}\sum_{i=1}^n\sum_{j=1}^n \|x_i-x_j\|^2\le \alpha \left(\frac{\log \left(\dd_X\sqrt{2}\right)}{1-\lambda_2(\A)}\right)^2\cdot \frac{1}{n}\sum_{i=1}^n\sum_{j=1}^n a_{ij} \|x_i-x_j\|^2.
$$
This estimate simplifies to give
\begin{equation}\label{eq:dx lower}
\dd_X\ge \frac{1}{\sqrt{2}} \exp\left(\frac{1-\lambda_2(\A)}{\sqrt{\alpha n}}\bigg(\frac{\sum_{i=1}^n\sum_{j=1}^n\|x_i-x_j\|^2}{\sum_{i=1}^n\sum_{j=1}^na_{ij}\|x_i-x_j\|^2}\bigg)^{\frac12}\right).
\end{equation}
The desired estimate~\eqref{eq:dimension lower intro}  (with $C=2/\sqrt{\alpha}$) now follows because $\dd_X\le \sqrt{\dim(X)}$ by~\cite{Joh48}.
\end{proof}

\begin{remark}\label{rem:must change norm}
Suppose that $\G=([n],E_\G)$ is a Cayley graph of a finite group such that $\lambda_2(\G)=1-\Omega(1)$. The metric space $([n],d_\G)$  embeds with distortion $\diam(\G)$ into $\ell_2^{n-1}$ by considering any bijection between $[n]$ and the vertices of the $n$-simplex. There is therefore no a priori reason why it wouldn't be possible to embed $([n],d_\G)$ with distortion $O(1)$ into some normed space $X$ whose Banach--Mazur distance from a Hilbert space is at least a sufficiently large multiple of $\diam (\G)$. But this is not so if $\diam(\G)$ is sufficiently large. Indeed, recalling Remark~\ref{rem:quotients}, it follows from~\eqref{eq:dx lower} that any embedding of  $([n],d_\G)$ into $X$ incurs distortion that is at least a universal constant multiple of $\diam(\G)/\log(2\dd_X)$. Thus, even if we allow  $\dd_X$ to be as large as $\diam(\G)^{O(1)}$, then any embedding of $([n],d_\G)$ into $X$ incurs distortion that is at least a universal constant multiple of $\diam(\G)/\log\diam(\G)$. Also, if $\diam(\G)\gtrsim \log n$ (e.g., if $\G$ has bounded degree) then this means that any embedding of  $([n],d_\G)$ into $X$ incurs distortion that is at least a universal constant multiple of $(\log n)/\log(2\dd_X)$ and, say, even if we allow  $\dd_X$ to be as large as $(\log n)^{O(1)}$, then any embedding of $([n],d_\G)$ into $X$ incurs distortion that is at least a universal constant multiple of $(\log n)/\log\log n$.
\end{remark}

\subsection{Proof of Theorem~\ref{th:isomorph}} We have seen that in order to prove Theorem~\ref{thm:main matrix intro} it suffices to prove Theorem~\ref{th:isomorph}. In order to do so, we shall first describe several ingredients that appear in its proof.

\subsubsection{Uniform convexity and smoothness}\label{sec:smoothness} Suppose that $(X,\|\cdot\|)$ is a normed space and fix $p,q>0$ satisfying $1\le p\le2\le q$. Following Ball, Carlen and Lieb~\cite{BCL94}, the $p$-smoothness constant of $X$, denoted $\mathscr{S}_p(X)$, is the infimum over those $S>0$ such that
\begin{equation}\label{eq:def smoothness}
\forall\, x,y\in X,\qquad \|x+y\|^p+\|x-y\|^p\le 2\|x\|^p+2S^p\|y\|^p.
\end{equation}
(If no such $S$ exists, then define $\mathscr{S}_p(X)=\infty$.) By the triangle, inequality we always have $\mathscr{S}_1(X)=1$. The $q$-convexity constant of $X$, denoted $\mathscr{K}_q(X)$, is the infimum over those $K>0$ such that
$$
\forall\, x,y\in X,\qquad 2\|x\|^q+\frac{2}{K^q}\|y\|^q\le \|x+y\|^q+\|x-y\|^q.
$$
(As before, if no such $K$ exists, then define $\mathscr{K}_q(X)=\infty$.) We refer to~\cite{BCL94} for the relation of these parameters to more traditional moduli of uniform convexity and smoothness that appear in the literature. It is beneficial to work with the quantities $\mathscr{S}_p(X), \mathscr{K}_q(X)$ rather than the classical moduli because they are well-behaved with respect to basic operations, an example of which is the duality $\mathscr{K}_{p/(p-1)}(X^*)=\mathscr{S}_p(X)$, as shown in~\cite{BCL94}. Another example that is directly relevant to us is their especially clean behavior under complex interpolation, as derived in Section~\ref{sec:interpolation} below.

\subsubsection{Complexification}\label{sec:complexification} All of the above results were stated for normed spaces over the real numbers, but in the ensuing proofs we need to consider normed spaces over the complex numbers. We do so through the use of the standard notion of complexification. Specifically, for every normed space $(X,\|\cdot\|_X)$ over $\R$ one associates as follows a normed space $(X_\C,\|\cdot\|_{X_\C})$ over $\C$. The underlying vector space is $X_\C=X\times X$, which is viewed as a vector space over $\C$ by setting $(\alpha+\beta i)(x,y)=(\alpha x-\beta y,\beta x+\alpha y)$ for every $\alpha,\beta\in \R$ and $x,y\in X$. The norm on $X_\C$ is given by
\begin{equation}\label{eq:def complexification}
\forall\, x,y\in X,\qquad \|(x,y)\|_{X_\C}=\bigg(\frac{1}{\pi}\int_0^{2\pi}\big\|(\cos\theta)x-(\sin\theta) y\big\|^2_X\ud \theta\bigg)^{\frac12}.
\end{equation}

The normalization in~\eqref{eq:def complexification} ensures that $x\mapsto (x,0)$ is an isometric embedding of $X$  into $X_\C$. It is straightforward to check that   for every $n\in \N$ and every symmetric stochastic matrix $\A\in M_n(\R)$ we have $\gamma(\A,\|\cdot\|_X^2)=\gamma(\A,\|\cdot\|_{X_\C}^2)$. Also, $\mathscr{S}_2(X_\C)=\mathscr{S}_2(X)$ and $\mathscr{K}_2(X_\C)=\mathscr{K}_2(X)$. When $p\in (1,2)$ and $q\in (2,\infty)$ we have $\mathscr{S}_p(X_\C)\asymp\mathscr{S}_p(X)$ and $\mathscr{K}_q(X_\C)\asymp\mathscr{K}_q(X)$; if one were to allow the implicit constants in these asymptotic  equivalences to depend on $p,q$ then this follows from the results of~\cite{FP74,Fig76,BCL94}, and the fact that these constants can actually be taken to be universal follows from carrying out the relevant arguments with more care, as done in~\cite{Nao12-azuma,MN14} (see specifically Lemma~6.3 and Corollary 6.4 of~\cite{MN14}). Finally, we have $\dd_{X_\C}=\dd_{X}$.

\subsubsection{Complex interpolation}\label{sec:interpolation} We very briefly recall Calder\'on's vector-valued complex interpolation method~\cite{Cal64}; see Chapter~4 of the monograph~\cite{BL76} for an extensive treatment. A pair of complex Banach spaces $(Y,\|\cdot\|_Y), (Z,\|\cdot\|_Z)$ is said to be compatible if they are both linearly embedded into a complex linear space $W$ with $Y+Z=W$. The space $W$ is a complex Banach space under the norm $\|w\|_W=\inf\{\|y\|_Y+\|z\|_Z:\ y+z=w\}$. Let $\mathcal{F}(Y,Z)$ denote the space of all bounded continuous functions $\psi:\{\zeta\in \C:\ 0\le \Re(\zeta)\le 1\}\to  W$ that are analytic on the open strip $\{\zeta\in \C:\ 0<\Re(\zeta)<1\}$. To every $\theta\in [0,1]$ one associates a Banach space $[Y,Z]_\theta$ as follows. The underlying vector space is $\{\psi(\theta):\ \psi\in \mathcal{F}(Y,Z)\}$, and the norm of $w\in [Y,Z]_\theta$ is given by $\|w\|_{[Y,Z]_\theta}=\inf_{\{\psi\in \mathcal{F}(Y,Z):\ \psi(\theta)=w\}}\max\{\sup_{t\in \R} \|\psi(ti)\|_{Y},\sup_{t\in \R}\|\psi(1+ti)\|_Z\}$. This turns $[Y,Z]_\theta$ into a Banach space, and we have $[Y,Z]_0=Y, [Y,Z]_1=Z$. Also, $[Y,Y]_\theta=Y$ for every $\theta\in [0,1]$.

Calder\'on's vector-valued version~\cite{Cal64} of the Riesz--Thorin theorem~\cite{Rie27,Tho48} asserts that if $(Y,\|\cdot\|_Y),(Z,\|\cdot\|_Z)$ and $(U,\|\cdot\|_U),(V,\|\cdot\|_V)$ are two compatible pairs of complex Banach spaces and $T:Y\cap Z\to U\cap V$ is a linear operator that extends to a bounded linear operator from $(Y,\|\cdot\|_Y)$ to $(U,\|\cdot\|_U)$ and from $(Z,\|\cdot\|_Z)$ to $(V,\|\cdot\|_V)$, then the following operator norm bounds hold true.
\begin{equation}\label{eq:riesz thorin}
\forall\, \theta\in [0,1],\qquad \|T\|_{[Y,Z]_\theta\to [U,V]_\theta}\le \|T\|_{Y\to U}^{1-\theta}\|T\|_{Z\to V}^\theta.
\end{equation}
The ensuing proof of Theorem~\ref{th:isomorph} uses the interpolation inequality~\eqref{eq:riesz thorin} four times (one of which is within the proof of a theorem that we shall quote from~\cite{Nao14}; see Theorem~\ref{thm:quote nao14} below). We shall now proceed to derive some preparatory estimates that will be needed in what follows.

 For every $p\ge 1$, every complex Banach space $(Z,\|\cdot\|_Z)$, and every weight $\omega:\{1,2\}\to [0,\infty)$ on the $2$-point set $\{1,2\}$, we denote (as usual) by $L_p(\omega;Z)$ the space $Z\times Z$ equipped with the norm that is given by setting $\|(a,b)\|_{L_p(\omega;Z)}^p=\omega(1)\|a\|_Z^p+\omega(2) \|b\|_Z^p$ for every $a,b\in Z$.

 If $(Y,\|\cdot\|_Y),(Z,\|\cdot\|_Z)$ is a compatible pair of complex Banach spaces then by Calder\'on's vector-valued version of Stein's interpolation theorem~\cite[Theorem~2]{Ste56} (see part(i) of $\S13.6$ in~\cite{Cal64} or Theorem~5.3.6 in~\cite{BL76}), for every $p,q\in [1,\infty]$, $\theta\in [0,1]$ and $\omega,\tau:\{1,2\}\to [0,\infty)$ we have
 \begin{equation}\label{eq:stein weiss}
 \left[L_p(\omega;Y),L_q(\tau;Z)\right]_\theta= L_{r}\Big(\omega^{\frac{1-\theta}{p}}\tau^{\frac{\theta }{q}};[Y,X]_\theta\Big),\quad \mathrm{where}\quad  r=\frac{pq}{\theta p+(1-\theta) q}.
 \end{equation}
 The  equality in~\eqref{eq:stein weiss} is in the sense of isometries, i.e., the norms on both sides coincide.

 Suppose that $p_1,p_2\in [1,2]$ and that the smoothness constants $\mathscr{S}_{p_1}(Y),\mathscr{S}_{p_2}(Z)$ are finite. Fix $S_1>\mathscr{S}_{p_1}(Y)$ and $S_2>\mathscr{S}_{p_2}(Z)$. Then by~\eqref{eq:def smoothness} we have
 \begin{equation}\label{eq:S1}
 \forall\, y_1,y_2\in Y,\qquad \|y_1+y_2\|_Y^{p_1}+\|y_1-y_2\|_Y^{p_1}\le 2\|y_1\|_Y^{p_1}+2S_1^{p_1}\|y_2\|_Y^{p_1},
 \end{equation}
 and
  \begin{equation}\label{eq:S2}
 \forall\, z_1,z_2\in Z,\qquad \|z_1+z_2\|_Z^{p_2}+\|z_1-z_2\|_Z^{p_2}\le 2\|z_1\|_Z^{p_2}+2S_2^{p_2}\|z_2\|_Z^{p_2}.
 \end{equation}
For every $S>0$ and $p\ge 1$ define $\omega(S,p):\{1,2\}\to (0,\infty)$ by $\omega(S,p)(1)=2$ and $\omega(S,p)(2)=2S^p$. Also, denote the constant function $\1_{\{1,2\}}$ by $\tau:\{1,2\}\to (0,\infty)$, i.e., $\tau(1)=\tau(2)=1$. With this notation, if we consider the linear operator $T:(Y+Z)\times (Y+Z)\to (Y+Z)\times (Y+Z)$ that is given by setting $T(w_1,w_2)=(w_1+w_2,w_1-w_2)$ for every $w_1,w_2\in Y+Z$, then
\begin{equation}\label{eq:p1 p2 norms}
\|T\|_{L_{p_1}(\omega(S_1,p_1);Y)\to L_{p_1}(\tau;Y)}\stackrel{\eqref{eq:S1}}{\le} 1\qquad\mathrm{and}\qquad \|T\|_{L_{p_2}(\omega(S_2,p_2);Z)\to L_{p_2}(\tau;Z)}\stackrel{\eqref{eq:S2}}{\le} 1.
\end{equation}
Denoting $r=p_1p_2/(\theta p_1+(1-\theta) p_2)$,  observe that $\omega(S_1,p_1)^{(1-\theta)/r}\omega(S_2,p_2)^{\theta/r}=\omega(S_1^{1-\theta}S_2^\theta,r)$. Hence, by~\eqref{eq:stein weiss} we have $[L_{p_1}(\omega(S_1,p_1);Y),L_{p_2}(\omega(S_2,p_2);Z)]_\theta=L_r(\omega(S_1^{1-\theta}S_2^\theta,r);[Y,Z]_\theta)$ and also $[L_{p_1}(\tau;Y);L_{p_2}(\tau;Z)]_\theta=L_r(\tau,[Y,Z]_\theta)$. In combination with~\eqref{eq:riesz thorin} and~\eqref{eq:p1 p2 norms}, these identities imply that the norm of $T$ as an operator from $L_r(\omega(S_1^{1-\theta}S_2^\theta,r);[Y,Z]_\theta)$ to $L_r(\tau,[Y,Z]_\theta)$ is at most $1$.  In other words, every $w_1,w_2\in [Y,Z]_\theta$ satisfy
$$
 \|w_1+w_2\|_{[Y,Z]_\theta}^{r}+\|w_1-w_2\|_{[Y,Z]_\theta}^{r}\le 2\|w_1\|_{[Y,Z]_\theta}^{r}+2\Big(S_1^{1-\theta}S_2^\theta\Big)^{r}\|w_2\|_{[Y,Z]_\theta}^{r}.
$$
Since $S_1$ and $S_2$ can be arbitrarily close to  $\mathscr{S}_{p_1}(Y)$ and $\mathscr{S}_{p_2}(Z)$, respectively,  we conclude that
\begin{equation}\label{eq:S interpolation}
\mathscr{S}_{\frac{p_1p_2}{\theta p_1+(1-\theta)p_2}}\big([Y,Z]_\theta\big)\le \mathscr{S}_{p_1}(Y)^{1-\theta} \mathscr{S}_{p_2}(Z)^\theta.
\end{equation}
By an analogous argument, if $q_1,q_2\ge 2$ and the convexity constants $\mathscr{K}_{q_1}(Y), \mathscr{K}_{q_2}(Z)$ are finite, then
\begin{equation}\label{eq:K interpolation}
\mathscr{K}_{\frac{q_1q_2}{\theta q_1+(1-\theta)q_2}}\big([Y,Z]_\theta\big)\le \mathscr{K}_{q_1}(Y)^{1-\theta} \mathscr{K}_{q_2}(Z)^\theta.
\end{equation}

\begin{remark} If one considers the traditional moduli of uniform convexity and smoothness (see e.g.~\cite{LT79} for the definitions), then interpolation statements that are analogous to~\eqref{eq:S interpolation}, \eqref{eq:K interpolation} are an old result of Cwikel and Reisner~\cite{CR82}, with the difference that~\cite{CR82} involves implicit constants that depend on $p_1,p_2,q_1,q_2$. By~\cite{BCL94}, this statement of~\cite{CR82} yields the estimates~\eqref{eq:S interpolation}, \eqref{eq:K interpolation} with additional factors in the right hand side that depend on  $p_1,p_2,q_1,q_2$. For our present purposes, i.e., for the proof of Theorem~\ref{th:isomorph}, it is important to obtain universal constants here. We believe that by carrying out the proofs in~\cite{CR82} with more care this could be achieved, but by working instead with the quantities $\mathscr{S}_p(\cdot), \mathscr{K}_q(\cdot)$ through the above simple (and standard) interpolation argument, we circumvented the need to do this and obtained the clean interpolation statements~\eqref{eq:S interpolation}, \eqref{eq:K interpolation}.
\end{remark}

Next, suppose that $(X,\|\cdot\|)$ is a Banach space over $\R$ with $\dd_X<\infty$. Fix $\dd>\dd_X$ and a Hilbertian norm $|\cdot|:X\to [0,\infty)$ that satisfies~\eqref{eq:d-isomorph}.   Denote by $H$ the Hilbert space that is induced by $|\cdot|$. Consider the complexifications $X_\C$ and $H_\C$. Observe that by~\eqref{eq:def complexification} and~\eqref{eq:d-isomorph} we have
\begin{equation}\label{eq:lift to complexification}
\forall\, x,y\in X,\qquad \|(x,y)\|_{H_\C}=\sqrt{|x|^2+|y^2|}\quad\mathrm{and}\quad \|(x,y)\|_{H_\C}\le \|(x,y)\|_{X_\C}\le \dd \|(x,y)\|_{H_\C}.
\end{equation}
Since $X_\C$ and $H_\C$ are isomorphic Banach space with the same underlying vector space (over $\C$), they are a compatible, and therefore for every $\theta\in [0,1]$ we can consider the complex interpolation space $[H_\C,X_\C]_{\theta}$. The formal identity operator $\mathsf{I}_{X\times X}:X\times X\to X\times X$ satisfies
\begin{equation}\label{eq:norm bounds}
\|\mathsf{I}_{X\times X}\|_{X_\C\to X_\C}\le 1,\quad \|\mathsf{I}_{X\times X}\|_{H_\C\to H_\C}\le 1,\quad \|\mathsf{I}_{X\times X}\|_{X_\C\to H_\C}\le 1,\quad \|\mathsf{I}_{X\times X}\|_{H_\C\to X_\C}\le \dd.
\end{equation}
The first two inequalities in~\eqref{eq:norm bounds} are tautological, and the final two inequalities in~\eqref{eq:norm bounds} are a consequence of the inequalities in~\eqref{eq:lift to complexification}. Hence,
$$
\|\mathsf{I}_{X\times X}\|_{[X_\C,H_\C]_\theta\to X_\C}=\|\mathsf{I}_{X\times X}\|_{[X_\C,H_\C]_\theta\to [X_\C,X_\C]_\theta}\stackrel{\eqref{eq:riesz thorin}}{\le} \|\mathsf{I}_{X\times X}\|_{X_\C\to X_\C}^{1-\theta}\|\mathsf{I}_{X\times X}\|_{H_\C\to X_\C}^{\theta}\stackrel{\eqref{eq:norm bounds}}{\le} \dd^\theta,
$$
and
$$
\|\mathsf{I}_{X\times X}\|_{ X_\C\to [X_\C,H_\C]_\theta}=\|\mathsf{I}_{X\times X}\|_{ [X_\C,X_\C]_\theta\to [X_\C,H_\C]_\theta}\stackrel{\eqref{eq:riesz thorin}}{\le} \|\mathsf{I}_{X\times X}\|_{X_\C\to X_\C}^{1-\theta}\|\mathsf{I}_{X\times X}\|_{X_\C\to H_\C}^{\theta}\stackrel{\eqref{eq:norm bounds}}{\le} 1.
$$
These two estimates can be restated as follows.
\begin{equation}\label{eq:distance to interpolant}
\forall\, x,y\in X,\qquad \|(x,y)\|_{[X_\C,H_\C]_\theta}\le \|(x,y)\|_{X_\C}\le \dd^\theta \|(x,y)\|_{[X_\C,H_\C]_\theta}.
\end{equation}

In what follows, we will use crucially the following theorem, which relates nonlinear spectral gaps to complex interpolation and uniform smoothness; this result appears in~\cite{Nao14} as Corollary~4.7.

\begin{theorem}\label{thm:quote nao14} Suppose that $(\mathcal{H},\|\cdot\|_{\mathcal{H}})$ and $(Z,\|\cdot\|_Z)$ are a compatible pair of complex Banach spaces, with  $(\mathcal{H},\|\cdot\|_{\mathcal{H}})$ being a Hilbert space. Suppose that $q\in [1,2]$ and $\theta\in (0,1]$. For every $n\in \N$ and every symmetric stochastic matrix $\A\in M_n(\R)$ we have
\begin{equation}\label{eq:XH}
\gamma\left(\A,\|\cdot\|_{[Z,\mathcal{H}]_\theta}^2\right)\lesssim \frac{\mathscr{S}_q\left([Z,\mathcal{H}]_\theta\right)^2}{\theta^{\frac{2}{q}}\left(1-\lambda_2(\A)\right)^{\frac{2}{q}}}.
\end{equation}
\end{theorem}

We note in passing that in~\cite{Nao14} (specifically, in the statement of~\cite[Theorem~4.5]{Nao14}) there is the following misprint: \eqref{eq:XH} is stated there for the transposed interpolation space  $[\mathcal{H},X]_\theta$ rather than the correct space $[X,\mathcal{H}]_\theta$ as above. This misprint is not confusing when one reads~\cite{Nao14} in context rather the statement of~\cite[Theorem~4.5]{Nao14} in isolation (e.g., clearly~\eqref{eq:XH} should not deteriorate   as the interpolation space approaches the Hilbert space $\mathcal{H}$). Also, the proof itself in~\cite{Nao14} deals with the correct interpolation space $[X,\mathcal{H}]_\theta$ throughout (see equation~(4.14) in~\cite{Nao14}).

\subsubsection{Completion of the proof of Theorem~\ref{th:isomorph}} Since for every Banach space $(X,\|\cdot\|)$ we have $\mathscr{S}_1(X)=1$, Theorem~\ref{th:isomorph} is the special case $p=1$ of the following more refined theorem.

\begin{theorem}\label{thm:Sp version with distance to l2} Fix $p\in [1,2]$ and suppose that $(X,\|\cdot\|)$ is a Banach space satisfying $\dd_X<\infty$ and $\mathscr{S}_p(X)<\infty$. For every $n\in \N$ and every symmetric stochastic matrix $\A=(a_{ij})\in M_n(\R)$, we have
\begin{equation}\label{eq:optimized cases p}
\gamma\big(\A,\|\cdot\|^2\big)\lesssim\left\{\begin{array}{ll}\frac{\dd_X^2}{1-\lambda_2(\A)}&\mathrm{if}\ \ \dd_X^p\left(1-\lambda_2(\A)\right)^{1-\frac{p}{2}}\le e\mathscr{S}_p(X)^p,\\
\frac{\mathscr{S}_p(X)^2}{\left(1-\lambda_2(\A)\right)^{\frac{2}{p}}}
\left(\log\left(\frac{\dd_X^p\left(1-\lambda_2(\A)\right)^{1-\frac{p}{2}}}{\mathscr{S}_p(X)^p}\right)\right)^{\frac{2}{p}}&\mathrm{if}\ \ \dd_X^p\left(1-\lambda_2(\A)\right)^{1-\frac{p}{2}}\ge e\mathscr{S}_p(X)^p.\end{array}\right.
\end{equation}
\end{theorem}

\begin{proof}
 Fix $\dd>\dd_X$ and $\theta\in (0,1]$. Consider a Hilbertian norm $|\cdot|:X\to [0,\infty)$ that satisfies~\eqref{eq:d-isomorph} and denote by $H$ the Hilbert space that is induced by $|\cdot|$. As we explained in Section~\ref{sec:complexification}, the complexification $X_\C$ satisfies $\mathscr{S}_p(X_\C)\asymp  \mathscr{S}_p(X)$. Also, by the parallelogram identity, the complex Hilbert space $H_\C$ satisfies $\mathscr{S}_2(H_\C)=1$. Hence, by~\eqref{eq:S interpolation} with $Y=X_\C$, $Z=H_\C$, $p_1=p$ and $p_2=2$,
$$
\mathscr{S}_{\frac{2p}{\theta p+2(1-\theta)}}\left(\left[X_\C,H_\C\right]_\theta\right)\le \mathscr{S}_p(X_\C)^{1-\theta}\lesssim \mathscr{S}_p(X)^{1-\theta}.
$$
We may therefore apply Theorem~\ref{thm:quote nao14}  with $q=(2p)/(\theta p+2(1-\theta))$ to deduce that
\begin{equation}\label{eq:gamma on interpolant before passing back}
\gamma\left(\A,\|\cdot\|_{[X_\C,H_\C]_\theta}^2\right)\lesssim \frac{\mathscr{S}_p(X)^{2(1-\theta)}}{\theta^{\theta+\frac{2(1-\theta)}{p}}\left(1-\lambda_2(\A)\right)^{\theta+\frac{2(1-\theta)}{p}}}\asymp \frac{\mathscr{S}_p(X)^{2(1-\theta)}}{\theta^{\frac{2}{p}}\left(1-\lambda_2(\A)\right)^{\theta+\frac{2(1-\theta)}{p}}}.
\end{equation}
By the definition of $\gamma\left(\A,\|\cdot\|_{[X_\C,H_\C]_\theta}^2\right)$, for every $(x_1,y_1),\ldots,(x_n,y_n)\in X\times X$ we have
$$
\frac{1}{n^2}\sum_{i=1}^n\sum_{j=1}^n \|(x_i-x_j,y_i-y_j)\|_{[X_\C,H_\C]_\theta}^2\le \frac{\gamma\left(\A,\|\cdot\|_{[X_\C,H_\C]_\theta}^2\right)}{n}\sum_{i=1}^n\sum_{j=1}^n a_{ij}\|(x_i-x_j,y_i-y_j)\|_{[X_\C,H_\C]_\theta}^2.
$$
By~\eqref{eq:distance to interpolant}, this implies that
$$
\frac{1}{n^2}\sum_{i=1}^n\sum_{j=1}^n \|(x_i-x_j,y_i-y_j)\|_{X_\C}^2\le \frac{\dd^{2\theta}\gamma\left(\A,\|\cdot\|_{[X_\C,H_\C]_\theta}^2\right)}{n}\sum_{i=1}^n\sum_{j=1}^n a_{ij}\|(x_i-x_j,y_i-y_j)\|_{X_\C}^2.
$$
Due to~\eqref{eq:gamma on interpolant before passing back} and  because $X$ is isometric to a subspace of $X_\C$, this implies that
\begin{equation}\label{eq:interpolated with p smoothness}
\forall\, \theta\in (0,1],\qquad \gamma\big(\A,\|\cdot\|^2\big)\lesssim  \frac{\dd_X^{2\theta}\mathscr{S}_p(X)^{2(1-\theta)}}{\theta^{\frac{2}{p}}\left(1-\lambda_2(\A)\right)^{\theta+\frac{2(1-\theta)}{p}}}.
\end{equation}

If $\dd_X^p\left(1-\lambda_2(\A)\right)^{1-p/2}\le e\mathscr{S}_p(X)^p$, then by substituting $\theta=1$ into~\eqref{eq:interpolated with p smoothness} we obtain the first range of~\eqref{eq:optimized cases p}. When $\dd_X^p\left(1-\lambda_2(\A)\right)^{1-p/2}> e\mathscr{S}_p(X)^p$ the following value of $\theta$ minimizes the right hand side of~\eqref{eq:interpolated with p smoothness} and belongs to the interval $(0,1]$.
$$
\theta_{\mathrm{opt}}\eqdef\frac{1}{\log\left(\frac{\dd_X^p\left(1-\lambda_2(\A)\right)^{1-\frac{p}{2}}}{\mathscr{S}_p(X)^p}\right)}.
$$
A substitution of $\theta_{\mathrm{opt}}$ into~\eqref{eq:interpolated with p smoothness} yields an estimate that simplifies to give the second range of~\eqref{eq:optimized cases p}.
\end{proof}

\section{Worst-case to average-case logarithmic improvement of Euclidean distortion}\label{sec:average dist}

Thus far we used only the case $p=1$ of Theorem~\ref{thm:Sp version with distance to l2}. The purpose of this short section is to establish the following consequence of the case $p=2$ of  Theorem~\ref{thm:Sp version with distance to l2}.

\begin{theorem}\label{thm:worst to average} There is a universal constant $\beta>0$ with the following property.
Suppose that $D\ge 1$ and let $(X,\|\cdot\|_X)$ be a normed space that embeds with bi-Lipschitz distortion $D$ into a Hilbert space. Then every finite subset of $X$ embeds with average distortion $\beta\mathscr{S}_2(X)^3\log (2D)$ into a Hilbert space.
\end{theorem}

A qualitative rephrasing of  Theorem~\ref{thm:worst to average} is the following somewhat curious assertion that is  nevertheless quite striking. If $X$ is $2$-convex, then the fact that every finite subset of $X$ admits an embedding into a Hilbert space with a worst-case pairwise requirement (i.e., bi-Lipschitz), automatically implies that every finite subset of $X$ admits an  embedding into a Hilbert space with an average-case  guarantee that grows at most like the logarithm  of the initial bi-Lipschitz distortion. It follows from the proof of~\cite[Lemma~1.12]{Nao14} that this phenomenon does not hold true under the weaker assumption $\mathscr{S}_p(X)<\infty$ for some $p\in [1,2)$. We do not know the extent to which the bound on the average distortion in Theorem~\ref{thm:worst to average} is sharp; it seems likely that the dependence on $\mathscr{S}_2(X)$ could be improved here. It is even conceivable that any finite subset of a Banach space $X$ with $\mathscr{S}_2(X)<\infty$ embeds with average distortion $\Psi(\mathscr{S}_2(X))$ into a Hilbert space, where $\Psi(\mathscr{S}_2(X))>0$ is a finite quantity that may depend only on $\mathscr{S}_2(X)$. We have no reason to conjecture that this is so, but if it were valid then it would be a remarkable geometric statement. We can show that this does hold true when $X$ is a $2$-convex Banach lattice, as explained in Remark~\ref{rem:lattice} below.

In the proof of Theorem~\ref{thm:worst to average}  we shall use the following standard notation related to the Lipschitz extension problem. The {\em Lipschitz extension constant} of a pair of metric spaces $(\cM,d_\cM)$ and $(\cN,d_\cN)$,  denoted $\ee(\cM,\cN)$, is the infimum over those $K\in [1,\infty]$ such that for every $\Omega\subset \cM$ and $L\in (0,\infty)$,  every $L$-Lipschitz function $\f:\Omega\to \cN$ admits a $(KL)$-Lipschitz extension $\Phi:\cM\to \cN$.

\begin{proof}[Proof of Theorem~\ref{thm:worst to average}] By a classical differentiation argument (see e.g.~\cite[Corollary~7.10]{BL00}), the fact that $X$ embeds with bi-Lipschitz distortion $D$ into a Hilbert space implies that $\dd_X\le D$. Since (by substituting $x=0$ into~\eqref{eq:def smoothness}) we always have $\mathscr{S}_2(X)\ge 1$, the conclusion of Theorem~\eqref{thm:Sp version with distance to l2} implies that for every $n\in \N$, every symmetric stochastic matrix $\A\in M_n(\R)$ satisfies
\begin{equation}\label{eq:for duality}
\gamma\big(\A,\|\cdot\|^2\big)\lesssim \frac{\mathscr{S}_2(X)^2\log(2D)}{1-\lambda_2(\A)}.
\end{equation}

By Corollary~1.4 in~\cite{Nao14}, the validity of~\eqref{eq:for duality} for {\em every} symmetric stochastic matrix $\A$ implies that for every $x_1,\ldots,x_n\in X$ there exists a $1$-Lipschitz mapping $f:\{x_1,\ldots,x_n\}\to \ell_2$ such that
\begin{equation}\label{eq:sqrt log}
\bigg(\sum_{i=1}^n\sum_{j=1}^n \|f(x_i)-f(x_j)\|_{2}^2\bigg)^{\frac12}\gtrsim \frac{1}{\mathscr{S}_2(X)\sqrt{\log (2D)} }\bigg(\sum_{i=1}^n\sum_{j=1}^n \|x_i-x_j\|_{X}^2\bigg)^{\frac12}.
\end{equation}
The estimate~\eqref{eq:sqrt log} is already an assertion that is similar to the conclusion of Theorem~\ref{thm:worst to average}, except that one is preserving the average of squares of distances rather than the average of the distances themselves. And, the conclusion in~\eqref{eq:sqrt log} is even stronger, with the distortion being at most a constant multiple of $\mathscr{S}_2(X)\sqrt{\log (2D)}$ rather than the claimed bound of $\mathscr{S}_2(X)^3\log (2D)$.

One can pass from~\eqref{eq:sqrt log} to the usual notion of average distortion by combining the results of Section~7.4 in~\cite{Nao14} with the Lipschitz extension extension theorem of~\cite{Bal92,NPSS06}. Specifically, by inequality~(7.39) in~\cite{Nao14} it follows from the validity of~\eqref{eq:sqrt log} for every $n\in \N$ and every $x_1,\ldots, x_n\in X$  that there also exists a $1$-Lipschitz mapping $\f:X\to \ell_2$ such that
\begin{equation*}\label{eq:use ext}
\sum_{i=1}^n\sum_{j=1}^n \|\f(x_i)-\f(x_j)\|_{2}\gtrsim \frac{1}{\ee(X,\ell_2)\mathscr{S}_2(X)^2\log(2D)}\sum_{i=1}^n\sum_{j=1}^n \|x_i-x_j\|_{X}.
\end{equation*}
Since $\ee(X,\ell_2)\lesssim \mathscr{S}_2(X)$ by~\cite[Theorem~2.3]{NPSS06}, this concludes the proof of Theorem~\ref{thm:worst to average}. \end{proof}

\begin{remark}\label{rem:lattice} Suppose that $(\mathcal{H},\|\cdot\|_{\mathcal{H}})$ and $(Z,\|\cdot\|_Z)$ are a compatible pair of complex Banach spaces, with  $(\mathcal{H},\|\cdot\|_{\mathcal{H}})$ being a Hilbert space. Suppose also that $\theta\in (0,1]$ and that  $\mathscr{S}_2([Z,\mathcal{H}]_\theta)<\infty$. Then by Theorem~\ref{thm:quote nao14} for every $n\in \N$, every symmetric stochastic matrix $\A\in M_n(\R)$ satisfies
$$
\gamma\left(\A,\|\cdot\|_{[Z,\mathcal{H}]_\theta}^2\right)\lesssim \frac{\mathscr{S}_2\left([Z,\mathcal{H}]_\theta\right)^2}{\theta\left(1-\lambda_2(\A)\right)}.
$$
By combining~\cite[Corollary~1.4]{Nao14} with~\cite[(7.39)]{Nao14} and~\cite[Theorem~2.3]{NPSS06}, it follows from this that any finite subset of $[Z,\mathcal{H}]_\theta$ embeds into $\ell_2$ with average distortion $O(\mathscr{S}_2\left([Z,\mathcal{H}]_\theta\right)^3/\theta)$.

Suppose next that $X$ is a Banach lattice that satisfies $\mathscr{S}_2(X)<\infty$. An extrapolation theorem of Pisier~\cite{Pis79-lattices-interpolation} asserts that there exists $\theta\in (0,1]$ and Banach spaces $Z,\mathcal{H}$ as above such that the complexification $X_\C$ is isomorphic to $[Z,\mathcal{H}]_\theta$. An inspection of Pisier's proof of this theorem (see also Appendix~I in~\cite{BL00}), reveals that  both $\theta$ and the isomorphism constants can be taken to be bounded by a function of $\mathscr{S}_2(X)$ alone (with more work it is also possible to deduce explicit estimates here, but this is quite tedious and we will not include the argument). It therefore follows from the above discussion  that any finite subset of a $2$-convex  Banach lattice $X$ embeds with average distortion $\Psi(\mathscr{S}_2(X))$ into $\ell_2$, where $\Psi(\mathscr{S}_2(X))>0$ is a finite quantity that depends only on $\mathscr{S}_2(X)$.
\end{remark}

\section{Further discussion and open problems}\label{sec:discussion}

Given $D\ge 1$ and a metric space $(\cM,d)$, the definition in~\cite{LLR95} of the metric dimension $\dim_D(\cM)$  can be naturally refined by restricting the potential target spaces into which we wish to embed $\cM$. Specifically, suppose that $\mathscr{F}$ is a family of finite-dimensional normed spaces. Define $\dim_D(\cM;\mathscr{F})$ to be the minimum $k\in \N$ for which $\cM$ embeds with distortion $D$ into some $X\in \mathscr{F}$ of dimension at most $k$; if no such $X\in \mathscr{F}$ exists then denote $\dim_D(\cM;\mathscr{F})=\infty$. The quantity $\dim_D(\cM)$ is then equal to $\dim_D(\cM;\mathscr{F})$ when $\mathscr{F}$ consists of all finite-dimensional normed spaces.

The question of estimating $\dim_D(\cM;\mathscr{F})$ for various metric spaces $\cM$ and various families $\mathscr{F}$ of finite-dimensional normed spaces encompasses much of the research on dimensionality reduction, though not all of the work on dimensionality reduction belongs to this framework (examples of other directions include restrictions on the embeddings themselves, such as dimensionality reduction via linear mappings~\cite{JN10}, requiring guarantees that are not necessarily bi-Lipschitz~\cite{IN07,ABN11} and complexity theoretic issues~\cite{MS10}, among others). Notable special cases include $\mathscr{F}=\{\ell_2^k\}_{k=1}^\infty$, which was studied in~\cite{Mat90-bilip-low-euc}, or the Johnson--Lindenstrauss lemma~\cite{JL84} which asserts that
\begin{equation}\label{eq:JL}
\forall\, \e>0,\ \forall\, n\in \N,\qquad \sup_{\substack{\cM\subset \ell_2\\ |\cM|\le n}}\dim_{1+\e}\big(\cM; \{\ell_2^k\}_{k=1}^\infty\big)\asymp_\e \log n.
\end{equation}
The implicit dependence on $\e$ in~\eqref{eq:JL} was unknown for many years, but it has been very recently determined in~\cite{LN16} (up to universal constant factors, and except for a small range of values of $\e$ that tends to $0$ as $n\to \infty$); see also the slightly weaker result of Alon~\cite{Alo03}, as well as the related work of Alon and Klartag~\cite{AK16}. These works include examples of methods to prove the impossibility of dimensionality reduction in the almost-isometric Euclidean setting, that differ from the analytic methods that are used here (they rely on linear algebra and certain coding arguments).

Another important example is when $\mathscr{F}=\{\ell_1^k\}_{k=1}^\infty$ and the metric space $\cM$ is  a subset of $\ell_1$. In this setting, the best-known bounds are that there exist universal constants $c,C>0$ for which
\begin{equation}\label{eq:BC}
\forall\, D\ge 1,\ \forall\, n\in \N,\qquad n^{\frac{c}{D^2}}\le \sup_{\substack{\cM\subset \ell_1\\ |\cM|\le n}}\dim_{D}\big(\cM; \{\ell_1^k\}_{k=1}^\infty\big)\le \frac{Cn}{D}.
\end{equation}
The first inequality in~\eqref{eq:BC} is a famous theorem of Brinkman and Charikar~\cite{BC05}, whose proof devised a clever method for proving dimensionality reduction lower bounds through the use of linear programming; see also~\cite{ACNN11}. A different approach to the same lower bound, which we already discussed in Section~\ref{sec:overview},  is due to~\cite{LN04-diamond}. See also the elegant entropy-based approach of Regev~\cite{Reg13} to the dimensionality reduction impossibility results of~\cite{BC05,ACNN11}. The upper bound in~\eqref{eq:BC} is due to the forthcoming work~\cite{ANN16}. Of course, the gap between the bounds in~\eqref{eq:BC} is large, and it would be of great interest to determine the correct asymptotic behavior here.

When $\mathscr{F}=\{\ell_p^k\}_{k=1}^\infty$ for some $p\notin \{1,2,\infty\}$, remarkably little is known about the asymptotic behavior of the  supremum of $\dim_D(\cM;\{\ell_p^k\}_{k=1}^\infty)$ over all $\cM\subset \ell_p$ with $|\cM|\le n$. It is a tantalizing longstanding open problem to devise methods to address this question, i.e., for proving  either  positive results or impossibility results for dimensionality reduction in $\ell_p$, when $p\notin \{1,2,\infty\}$.

In the setting of dimensionality reduction for subsets of $\ell_1$ when the target space can  be a general normed space, we suspect that  the lower bound on $\dim_D(\cdot)$ of~\cite{Mat96} (as well as the bound obtained here) cannot occur for subsets of $\ell_1$. Specifically, we have the following conjecture.

\begin{conjecture}\label{conj:polylog l1}
Every $n$-point subset of $\ell_1$ embeds with distortion $O(1)$ into some normed space of dimension $(\log n)^{O(1)}$.
\end{conjecture}
See also Problem~3.5 in~\cite{Mat03}, that was stated by Mendel, where it is speculated that one could even obtain a normed space of dimension $O(\log n)$ in Conjecture~\ref{conj:polylog l1}. Due to~\cite{KLMN05}, it is conceivable that one could take the target space to be $\ell_1^{k}\oplus \ell_\infty^{k}$ for some $k=O(\log n)$. In the forthcoming work~\cite{ANN16}, an approach is devised for proving Conjecture~\eqref{conj:polylog l1}. Namely, Conjecture~\eqref{conj:polylog l1}  is established in~\cite{ANN16} under a yet unproven but plausible geometric hypothesis.

We indicated   a small part of the extensive literature on dimensionality reduction, as well as a few of the basic questions that remain open. A comprehensive  survey would exceed the scope of the present article, so we only state that by combining restrictions on the metric $\cM$ (e.g.~doubling metric spaces, planar graphs, series-parallel graphs, graphs of bounded bandwidth, trees, ultrametrics) with restrictions of the targets $\mathscr{F}$, leads to a  rich body of work, as well as fundamental unsolved questions. To the best of our knowledge,  we indicated in this section and in the Introduction all of the known methods for proving lower bounds on dimensionality reduction, with the present article contributing another such method. There is great need to obtain new approaches to these issues.

\bibliographystyle{alphaabbrvprelim}
\bibliography{almost-ext}

\newcommand{\etalchar}[1]{$^{#1}$}
\def\cprime{$'$} \def\cprime{$'$} \def\cprime{$'$} \def\cprime{$'$}
  \def\cprime{$'$} \def\cprime{$'$} \def\cprime{$'$} \def\cprime{$'$}
  \def\cprime{$'$} \def\cprime{$'$} \def\cprime{$'$}
\begin{thebibliography}{AdRRP92}
\expandafter\ifx\csname urlstyle\endcsname\relax
  \providecommand{\doi}[1]{doi:\discretionary{}{}{}#1}\else
  \providecommand{\doi}{doi:\discretionary{}{}{}\begingroup
  \urlstyle{rm}\Url}\fi

\bibitem[ABN11]{ABN11}
I.~Abraham, Y.~Bartal, and O.~Neiman.
\newblock Advances in metric embedding theory.
\newblock \emph{Adv. Math.}, 228(6):3026--3126, 2011.

\bibitem[ACNN11]{ACNN11}
A.~Andoni, M.~S. Charikar, O.~Neiman, and H.~L. Nguyen.
\newblock Near linear lower bound for dimension reduction in {$\ell_1$}.
\newblock In \emph{2011 {IEEE} 52nd {A}nnual {S}ymposium on {F}oundations of
  {C}omputer {S}cience---{FOCS} 2011}, pages 315--323. IEEE Computer Soc., Los
  Alamitos, CA, 2011.
\newblock \doi{10.1109/FOCS.2011.87}.

\bibitem[AdRRP92]{AR92}
J.~Arias-de Reyna and L.~Rodr{\'{\i}}guez-Piazza.
\newblock Finite metric spaces needing high dimension for {L}ipschitz
  embeddings in {B}anach spaces.
\newblock \emph{Israel J. Math.}, 79(1):103--111, 1992.

\bibitem[AFR85]{AFR85}
N.~Alon, P.~Frankl, and V.~R{\"{o}}dl.
\newblock Geometrical realization of set systems and probabilistic
  communication complexity.
\newblock In \emph{26th Annual Symposium on Foundations of Computer Science},
  pages 277--280. {IEEE} Computer Society, 1985.

\bibitem[AK16]{AK16}
N.~Alon and B.~Klartag.
\newblock Optimal compression of approximate {E}uclidean distances, 2016.
\newblock Preprint, available at \url{https://arxiv.org/abs/1610.00239}.

\bibitem[Alo03]{Alo03}
N.~Alon.
\newblock Problems and results in extremal combinatorics. {I}.
\newblock \emph{Discrete Math.}, 273(1-3):31--53, 2003.
\newblock EuroComb'01 (Barcelona).

\bibitem[ANN16]{ANN16}
A.~Andoni, A.~Naor, and O.~Neiman.
\newblock On isomorphic dimension reduction in $\ell_1$, 2016.
\newblock Preprint.

\bibitem[ANRW16]{ANRW16}
A.~Andoni, A.~Nikolov, I.~Razenshteyn, and E.~Waingarten.
\newblock Approximate near neighbors for general symmetric norms, 2016.
\newblock Preprint, available at \url{https://arxiv.org/pdf/1611.06222}.

\bibitem[AR94]{AR94}
N.~Alon and Y.~Roichman.
\newblock Random {C}ayley graphs and expanders.
\newblock \emph{Random Structures Algorithms}, 5(2):271--284, 1994.

\bibitem[Bal92]{Bal92}
K.~Ball.
\newblock Markov chains, {R}iesz transforms and {L}ipschitz maps.
\newblock \emph{Geom. Funct. Anal.}, 2(2):137--172, 1992.

\bibitem[Bal13]{Bal13}
K.~Ball.
\newblock The {R}ibe programme.
\newblock \emph{Ast\'erisque}, (352):Exp. No. 1047, viii, 147--159, 2013.
\newblock S{\'e}minaire Bourbaki. Vol. 2011/2012. Expos{\'e}s 1043--1058.

\bibitem[BC05]{BC05}
B.~Brinkman and M.~Charikar.
\newblock On the impossibility of dimension reduction in {$l_1$}.
\newblock \emph{J. ACM}, 52(5):766--788 (electronic), 2005.

\bibitem[BCL94]{BCL94}
K.~Ball, E.~A. Carlen, and E.~H. Lieb.
\newblock Sharp uniform convexity and smoothness inequalities for trace norms.
\newblock \emph{Invent. Math.}, 115(3):463--482, 1994.

\bibitem[BJL{\etalchar{+}}99]{BJLPS99}
S.~Bates, W.~B. Johnson, J.~Lindenstrauss, D.~Preiss, and G.~Schechtman.
\newblock Affine approximation of {L}ipschitz functions and nonlinear
  quotients.
\newblock \emph{Geom. Funct. Anal.}, 9(6):1092--1127, 1999.

\bibitem[BL76]{BL76}
J.~Bergh and J.~L{\"o}fstr{\"o}m.
\newblock \emph{Interpolation spaces. {A}n introduction}.
\newblock Springer-Verlag, Berlin-New York, 1976.
\newblock Grundlehren der Mathematischen Wissenschaften, No. 223.

\bibitem[BL00]{BL00}
Y.~Benyamini and J.~Lindenstrauss.
\newblock \emph{Geometric nonlinear functional analysis. {V}ol. 1}, volume~48
  of \emph{American Mathematical Society Colloquium Publications}.
\newblock American Mathematical Society, Providence, RI, 2000.

\bibitem[BLMN05]{BLMN05}
Y.~Bartal, N.~Linial, M.~Mendel, and A.~Naor.
\newblock On metric {R}amsey-type phenomena.
\newblock \emph{Ann. of Math. (2)}, 162(2):643--709, 2005.

\bibitem[Bou85]{Bou85}
J.~Bourgain.
\newblock On {L}ipschitz embedding of finite metric spaces in {H}ilbert space.
\newblock \emph{Israel J. Math.}, 52(1-2):46--52, 1985.

\bibitem[Bou86]{Bou86}
J.~Bourgain.
\newblock The metrical interpretation of superreflexivity in {B}anach spaces.
\newblock \emph{Israel J. Math.}, 56(2):222--230, 1986.

\bibitem[Cal64]{Cal64}
A.-P. Calder{\'o}n.
\newblock Intermediate spaces and interpolation, the complex method.
\newblock \emph{Studia Math.}, 24:113--190, 1964.

\bibitem[CR82]{CR82}
M.~Cwikel and S.~Reisner.
\newblock Interpolation of uniformly convex {B}anach spaces.
\newblock \emph{Proc. Amer. Math. Soc.}, 84(4):555--559, 1982.

\bibitem[Fig76]{Fig76}
T.~Figiel.
\newblock On the moduli of convexity and smoothness.
\newblock \emph{Studia Math.}, 56(2):121--155, 1976.

\bibitem[FP74]{FP74}
T.~Figiel and G.~Pisier.
\newblock S\'eries al\'eatoires dans les espaces uniform\'ement convexes ou
  uniform\'ement lisses.
\newblock \emph{C. R. Acad. Sci. Paris S\'er. A}, 279:611--614, 1974.

\bibitem[Fr{\'e}06]{Fre06}
M.~Fr{\'e}chet.
\newblock Sur quelques points du calcul fonctionnel.
\newblock \emph{Rend. Circ. Mat. Palermo}, 22:1--74, 1906.

\bibitem[HLW06]{HLW06}
S.~Hoory, N.~Linial, and A.~Wigderson.
\newblock Expander graphs and their applications.
\newblock \emph{Bull. Amer. Math. Soc. (N.S.)}, 43(4):439--561 (electronic),
  2006.

\bibitem[IN07]{IN07}
P.~Indyk and A.~Naor.
\newblock Nearest-neighbor-preserving embeddings.
\newblock \emph{ACM Trans. Algorithms}, 3(3):Art. 31, 12, 2007.

\bibitem[JL84]{JL84}
W.~B. Johnson and J.~Lindenstrauss.
\newblock Extensions of {L}ipschitz mappings into a {H}ilbert space.
\newblock In \emph{Conference in modern analysis and probability ({N}ew
  {H}aven, {C}onn., 1982)}, volume~26 of \emph{Contemp. Math.}, pages 189--206.
  Amer. Math. Soc., Providence, RI, 1984.

\bibitem[JLS87]{JLS87}
W.~B. Johnson, J.~Lindenstrauss, and G.~Schechtman.
\newblock On {L}ipschitz embedding of finite metric spaces in low-dimensional
  normed spaces.
\newblock In \emph{Geometrical aspects of functional analysis (1985/86)},
  volume 1267 of \emph{Lecture Notes in Math.}, pages 177--184. Springer,
  Berlin, 1987.
\newblock \doi{10.1007/BFb0078145}.

\bibitem[JN10]{JN10}
W.~B. Johnson and A.~Naor.
\newblock The {J}ohnson-{L}indenstrauss lemma almost characterizes {H}ilbert
  space, but not quite.
\newblock \emph{Discrete Comput. Geom.}, 43(3):542--553, 2010.

\bibitem[Joh48]{Joh48}
F.~John.
\newblock Extremum problems with inequalities as subsidiary conditions.
\newblock In \emph{Studies and {E}ssays {P}resented to {R}. {C}ourant on his
  60th {B}irthday, {J}anuary 8, 1948}, pages 187--204. Interscience Publishers,
  Inc., New York, N. Y., 1948.

\bibitem[Kal08]{Kal08}
N.~J. Kalton.
\newblock The complemented subspace problem revisited.
\newblock \emph{Studia Math.}, 188(3):223--257, 2008.

\bibitem[KLMN05]{KLMN05}
R.~Krauthgamer, J.~R. Lee, M.~Mendel, and A.~Naor.
\newblock Measured descent: a new embedding method for finite metrics.
\newblock \emph{Geom. Funct. Anal.}, 15(4):839--858, 2005.

\bibitem[KN06]{KN06}
S.~Khot and A.~Naor.
\newblock Nonembeddability theorems via {F}ourier analysis.
\newblock \emph{Math. Ann.}, 334(4):821--852, 2006.

\bibitem[LLR95]{LLR95}
N.~Linial, E.~London, and Y.~Rabinovich.
\newblock The geometry of graphs and some of its algorithmic applications.
\newblock \emph{Combinatorica}, 15(2):215--245, 1995.

\bibitem[LMN05]{LMN05}
J.~R. Lee, M.~Mendel, and A.~Naor.
\newblock Metric structures in {$L_1$}: dimension, snowflakes, and average
  distortion.
\newblock \emph{European J. Combin.}, 26(8):1180--1190, 2005.

\bibitem[LN04]{LN04-diamond}
J.~R. Lee and A.~Naor.
\newblock Embedding the diamond graph in {$L_p$} and dimension reduction in
  {$L_1$}.
\newblock \emph{Geom. Funct. Anal.}, 14(4):745--747, 2004.

\bibitem[LN16]{LN16}
K.~G. Larsen and J.~Nelson.
\newblock Optimality of the {J}ohnson-{L}indenstrauss lemma, 2016.
\newblock Preprint, available at \url{https://arxiv.org/abs/1609.02094}.

\bibitem[LPS88]{LPS88}
A.~Lubotzky, R.~Phillips, and P.~Sarnak.
\newblock Ramanujan graphs.
\newblock \emph{Combinatorica}, 8(3):261--277, 1988.

\bibitem[LT79]{LT79}
J.~Lindenstrauss and L.~Tzafriri.
\newblock \emph{Classical {B}anach spaces. {II}}, volume~97 of \emph{Ergebnisse
  der Mathematik und ihrer Grenzgebiete [Results in Mathematics and Related
  Areas]}.
\newblock Springer-Verlag, Berlin-New York, 1979.
\newblock ISBN 3-540-08888-1.
\newblock Function spaces.

\bibitem[Mar88]{Mar88}
G.~A. Margulis.
\newblock Explicit group-theoretic constructions of combinatorial schemes and
  their applications in the construction of expanders and concentrators.
\newblock \emph{Problemy Peredachi Informatsii}, 24(1):51--60, 1988.

\bibitem[Mat90]{Mat90-bilip-low-euc}
J.~Matou{\v{s}}ek.
\newblock Bi-{L}ipschitz embeddings into low-dimensional {E}uclidean spaces.
\newblock \emph{Comment. Math. Univ. Carolin.}, 31(3):589--600, 1990.

\bibitem[Mat92]{Mat92}
J.~Matou{\v{s}}ek.
\newblock Note on bi-{L}ipschitz embeddings into normed spaces.
\newblock \emph{Comment. Math. Univ. Carolin.}, 33(1):51--55, 1992.

\bibitem[Mat96]{Mat96}
J.~Matou{\v{s}}ek.
\newblock On the distortion required for embedding finite metric spaces into
  normed spaces.
\newblock \emph{Israel J. Math.}, 93:333--344, 1996.

\bibitem[Mat97]{Mat97}
J.~Matou{\v{s}}ek.
\newblock On embedding expanders into {$l_p$} spaces.
\newblock \emph{Israel J. Math.}, 102:189--197, 1997.

\bibitem[Mat02]{Mat02}
J.~Matou{\v{s}}ek.
\newblock \emph{Lectures on discrete geometry}, volume 212 of \emph{Graduate
  Texts in Mathematics}.
\newblock Springer-Verlag, New York, 2002.
\newblock ISBN 0-387-95373-6.
\newblock \doi{10.1007/978-1-4613-0039-7}.

\bibitem[Mat03]{Mat03}
J.~Matou\v{s}ek.
\newblock Open problems on embeddings of finite metric spaces, 2003.
\newblock Available at \url{http://kam.mff.cuni.cz/~matousek/metrop.ps}.

\bibitem[Mil64]{Mil64}
J.~Milnor.
\newblock On the {B}etti numbers of real varieties.
\newblock \emph{Proc. Amer. Math. Soc.}, 15:275--280, 1964.

\bibitem[MN13a]{MN13-convexity}
M.~Mendel and A.~Naor.
\newblock Markov convexity and local rigidity of distorted metrics.
\newblock \emph{J. Eur. Math. Soc. (JEMS)}, 15(1):287--337, 2013.

\bibitem[MN13b]{MN13}
M.~Mendel and A.~Naor.
\newblock Spectral calculus and {L}ipschitz extension for barycentric metric
  spaces.
\newblock \emph{Anal. Geom. Metr. Spaces}, 1:163--199, 2013.

\bibitem[MN14]{MN14}
M.~Mendel and A.~Naor.
\newblock Nonlinear spectral calculus and super-expanders.
\newblock \emph{Publ. Math. Inst. Hautes \'Etudes Sci.}, 119:1--95, 2014.

\bibitem[MN15]{MN15}
M.~Mendel and A.~Naor.
\newblock Expanders with respect to {H}adamard spaces and random graphs.
\newblock \emph{Duke Math. J.}, 164(8):1471--1548, 2015.

\bibitem[MS10]{MS10}
J.~Matou{\v{s}}ek and A.~Sidiropoulos.
\newblock Inapproximability for metric embeddings into {$\Bbb R^d$}.
\newblock \emph{Trans. Amer. Math. Soc.}, 362(12):6341--6365, 2010.

\bibitem[MW78]{MW78}
V.~D. Mil{\cprime}man and H.~Wolfson.
\newblock Minkowski spaces with extremal distance from the {E}uclidean space.
\newblock \emph{Israel J. Math.}, 29(2-3):113--131, 1978.

\bibitem[Nao12a]{Nao12}
A.~Naor.
\newblock An introduction to the {R}ibe program.
\newblock \emph{Jpn. J. Math.}, 7(2):167--233, 2012.

\bibitem[Nao12b]{Nao12-azuma}
A.~Naor.
\newblock On the {B}anach-space-valued {A}zuma inequality and small-set
  isoperimetry of {A}lon-{R}oichman graphs.
\newblock \emph{Combin. Probab. Comput.}, 21(4):623--634, 2012.

\bibitem[Nao14]{Nao14}
A.~Naor.
\newblock Comparison of metric spectral gaps.
\newblock \emph{Anal. Geom. Metr. Spaces}, 2:1--52, 2014.

\bibitem[NPSS06]{NPSS06}
A.~Naor, Y.~Peres, O.~Schramm, and S.~Sheffield.
\newblock Markov chains in smooth {B}anach spaces and {G}romov-hyperbolic
  metric spaces.
\newblock \emph{Duke Math. J.}, 134(1):165--197, 2006.

\bibitem[NR09]{NR09}
I.~Newman and Y.~Rabinovich.
\newblock Hard metrics from {C}ayley graphs of abelian groups.
\newblock \emph{Theory Comput.}, 5:125--134, 2009.

\bibitem[NS11]{NS11}
A.~Naor and L.~Silberman.
\newblock Poincar\'e inequalities, embeddings, and wild groups.
\newblock \emph{Compos. Math.}, 147(5):1546--1572, 2011.

\bibitem[Ost13]{Ost13}
M.~I. Ostrovskii.
\newblock \emph{Metric embeddings}, volume~49 of \emph{De Gruyter Studies in
  Mathematics}.
\newblock De Gruyter, Berlin, 2013.
\newblock Bilipschitz and coarse embeddings into Banach spaces.

\bibitem[Pis79]{Pis79-lattices-interpolation}
G.~Pisier.
\newblock Some applications of the complex interpolation method to {B}anach
  lattices.
\newblock \emph{J. Analyse Math.}, 35:264--281, 1979.

\bibitem[Pis10]{Pis10}
G.~Pisier.
\newblock Complex interpolation between {H}ilbert, {B}anach and operator
  spaces.
\newblock \emph{Mem. Amer. Math. Soc.}, 208(978):vi+78, 2010.

\bibitem[Pra02]{Pra02}
D.~Praetorius.
\newblock Remarks and examples concerning distance ellipsoids.
\newblock \emph{Colloq. Math.}, 93(1):41--53, 2002.

\bibitem[Rab08]{Rab08}
Y.~Rabinovich.
\newblock On average distortion of embedding metrics into the line.
\newblock \emph{Discrete Comput. Geom.}, 39(4):720--733, 2008.

\bibitem[Reg13]{Reg13}
O.~Regev.
\newblock Entropy-based bounds on dimension reduction in {$L_1$}.
\newblock \emph{Israel J. Math.}, 195(2):825--832, 2013.

\bibitem[Rib76]{Rib76}
M.~Ribe.
\newblock On uniformly homeomorphic normed spaces.
\newblock \emph{Ark. Mat.}, 14(2):237--244, 1976.

\bibitem[Rie27]{Rie27}
M.~Riesz.
\newblock Sur les maxima des formes bilin\'eaires et sur les fonctionnelles
  lin\'eaires.
\newblock \emph{Acta Math.}, 49(3-4):465--497, 1927.

\bibitem[Ste56]{Ste56}
E.~M. Stein.
\newblock Interpolation of linear operators.
\newblock \emph{Trans. Amer. Math. Soc.}, 83:482--492, 1956.

\bibitem[Tho48]{Tho48}
G.~O. Thorin.
\newblock Convexity theorems generalizing those of {M}. {R}iesz and {H}adamard
  with some applications.
\newblock \emph{Comm. Sem. Math. Univ. Lund [Medd. Lunds Univ. Mat. Sem.]},
  9:1--58, 1948.

\bibitem[Tho65]{Tho65}
R.~Thom.
\newblock Sur l'homologie des vari\'et\'es alg\'ebriques r\'eelles.
\newblock In \emph{Differential and {C}ombinatorial {T}opology ({A} {S}ymposium
  in {H}onor of {M}arston {M}orse)}, pages 255--265. Princeton Univ. Press,
  Princeton, N.J., 1965.

\bibitem[TJ89]{Tom89}
N.~Tomczak-Jaegermann.
\newblock \emph{Banach-{M}azur distances and finite-dimensional operator
  ideals}, volume~38 of \emph{Pitman Monographs and Surveys in Pure and Applied
  Mathematics}.
\newblock Longman Scientific \& Technical, Harlow; copublished in the United
  States with John Wiley \& Sons, Inc., New York, 1989.

\end{thebibliography}

 \end{document}